\documentclass[11pt]{amsart}
\usepackage{mathrsfs}
\usepackage{amsfonts}
\usepackage{amssymb}
\usepackage{amsxtra}
\usepackage{dsfont}
\usepackage{color}
\usepackage{graphicx}
\usepackage[english,polish]{babel}
\usepackage{enumerate}
\usepackage{mathtools}

\usepackage[dvipsnames]{xcolor}
\usepackage{tikz}
\usetikzlibrary{arrows}

\usepackage[compress, sort]{cite}

\usepackage{newtxmath}
\usepackage{newtxtext}

\usepackage[margin=2.5cm, centering]{geometry}
\usepackage[colorlinks,citecolor=blue,urlcolor=blue,bookmarks=true]{hyperref}
\hypersetup{
pdfpagemode=UseNone,
pdfstartview=FitH,
pdfdisplaydoctitle=true,
pdfborder={0 0 0}, 
pdftitle={Asymptotic behaviour of Christoffel--Darboux kernel via three-term recurrence relation II},
pdfauthor={Grzegorz Świderski and Bartosz Trojan},
pdflang=en-US
}

\usepackage[utf8]{inputenc}

\newcommand{\ZZ}{\mathbb{Z}}

\newcommand{\NN}{\mathbb{N}}
\newcommand{\RR}{\mathbb{R}}

\newcommand{\calR}{\mathcal{R}}
\newcommand{\calC}{\mathcal{C}}

\newcommand{\calX}{\mathcal{X}}

\newcommand{\calD}{\mathcal{D}}

\newcommand{\frakX}{\mathfrak{X}}
\newcommand{\frakB}{\mathfrak{B}}

\newcommand{\scrD}{\mathscr{D}}

\newcommand{\vphi}{\varphi}

\newcommand{\Id}{\operatorname{Id}}

\newcommand{\pl}[1]{\foreignlanguage{polish}{#1}}

\newcommand{\norm}[1]{\lVert {#1} \rVert}
\newcommand{\abs}[1]{\lvert {#1} \rvert}

\newcommand{\tr}{\operatorname{tr}}

\newcommand{\GL}{\operatorname{GL}}

\newcommand{\discr}{\operatorname{discr}}
\newcommand{\sym}{\operatorname{sym}}

\newcommand{\sinc}{\operatorname{sinc}}

\newcommand{\ud}{{\: \rm d}}

\newcommand{\supp}{\operatornamewithlimits{supp}}

\newtheorem{theorem}{Theorem}
\newtheorem{proposition}{Proposition}
\newtheorem{lemma}{Lemma}
\newtheorem{corollary}{Corollary}
\newtheorem{claim}{Claim}
\newtheorem{conjecture}{Conjecture}
\newtheorem*{theorem*}{Theorem}

\theoremstyle{definition}
\newtheorem{example}{Example}
\newtheorem{remark}{Remark}

\numberwithin{equation}{section}

\theoremstyle{plain}
\newcounter{thm}

\newtheorem{main_theorem}[thm]{Theorem}

\newcounter{cor}

\title[Asymptotic behaviour of Christoffel--Darboux kernel]
{Asymptotic behaviour of Christoffel--Darboux kernel via three-term recurrence relation II}

\author{Grzegorz Świderski}
\address{
	\pl{
		Grzegorz \'Swiderski \\
		Department of Mathematics \\
		KU Leuven \\
		Celestijnenlaan 200B box 2400 \\
		BE-3001 Leuven \\
		Belgium \&
		Mathematical Institute \\
		University of Wroc\l{}aw \\
		pl. Grunwaldzki 2/4 \\
		50-384 Wroc\l{}aw \\
		Poland
	}
}
\email{grzegorz.swiderski@kuleuven.be}

\author{Bartosz Trojan}
\address{
	\pl{
	Bartosz Trojan\\
	Institute of Mathematics\\
	Polish Academy of Sciences\\
	ul. \'Sniadeckich 8\\
	00-696 Warszawa\\
	Poland}
}
\email{btrojan@impan.pl}



\begin{document}
\selectlanguage{english}

\begin{abstract}
We study orthogonal polynomials with periodically modulated Jacobi parameters in the case when $0$ lies on the soft edge
of the spectrum of the corresponding periodic Jacobi matrix. We determine when the orthogonality measure is absolutely
continuous and we provide a constructive formula for it in terms of the limit of Tur\'an determinants. We next consider
asymptotics of the solutions of associated second order difference equation. Finally, we study scaling limits of the
Christoffel--Darboux kernel.
\end{abstract}

\maketitle

\section{Introduction}
Let $\mu$ be a probability measure on the real line with infinite support such that for every $n \in \NN_0$,
\[
	\text{the moments} \quad \int_{\RR} x^n \ud \mu(x) \quad \text{are finite}.
\]
Let $L^2(\RR, \mu)$ be the Hilbert space of square-integrable functions equipped with the scalar product
\[
	\langle f, g \rangle = \int_\RR f(x) \overline{g(x)} \ud \mu(x).
\]
By performing on the sequence of monomials $(x^n : n \in \NN_0)$ the Gram--Schmidt orthogonalization process 
one obtains the sequence of polynomials $(p_n : n \in \NN_0)$ satisfying
\begin{equation} \label{eq:1a}
	\langle p_n, p_m \rangle = \delta_{nm}
\end{equation}
where $\delta_{nm}$ is the Kronecker delta. Moreover, $(p_n : n \in \NN_0)$ satisfies the following recurrence relation
\begin{equation}
	\label{eq:100}
	\begin{aligned} 
	p_0(x) &= 1, \qquad p_1(x) = \frac{x - b_0}{a_0}, \\
	x p_n(x) &= a_n p_{n+1}(x) + b_n p_n(x) + a_{n-1} p_{n-1}(x), \qquad n \geq 1
	\end{aligned}
\end{equation}
where
\[
	a_n = \langle x p_n, p_{n+1} \rangle, \qquad
	b_n = \langle x p_n, p_n \rangle, \qquad n \geq 0.
\]
Notice that for every $n$, $a_n > 0$ and $b_n \in \RR$. The pair $(a_n)$ and $(b_n)$ is called the
\emph{Jacobi parameters}. One of the central object of this article is the \emph{Christoffel--Darboux} kernel $K_n$,
which is defined as
\[
	K_n(x, y) = \sum_{j=0}^n p_j(x) p_j(y).
\]
The classical topic in analysis is studying the asymptotic behavior of orthogonal polynomials $(p_n)$, which allows
to find the asymptotic of Christoffel--Darboux kernel. To motivate the interest in Christoffel--Darboux kernel
see surveys \cite{Lubinsky2016} and \cite{Simon2008}.

The case when the measure $\mu$ has compact support is well understood. For the asymptotics of the polynomials 
see e.g. the monograph \cite{StahlTotik1992} where the classical potential theory is the basic tool. Nowadays, the
so-called \emph{Riemann--Hilbert method} is commonly used to derive precise asymptotics of the orthogonal polynomials as
well as its Jacobi parameters, see e.g. \cite{Deift1999, Kuijlaars2004} and the book \cite{DeiftRH}. However, this method
demands stronger regularity conditions than those imposed in \cite{StahlTotik1992}. One of the most general result
concerning the Christoffel--Darboux kernel has been proven in \cite{Totik2009}.  Namely, if $I$ is an open
interval contained in $\supp(\mu)$, so that $\mu$ is absolutely continuous on $I$ with continuous positive density
$\mu'$, then
\begin{equation} 
	\label{eq:97}
	\lim_{n \to \infty} \frac{1}{n} K_n \Big( x + \frac{u}{n}, x + \frac{v}{n} \Big) = 
	\frac{\omega'(x)}{\mu'(x)} \sinc \big( (u-v) \pi \omega'(x) \big)
\end{equation}
locally uniformly with respect to $x \in I$ and $u, v \in \RR$, provided that $\mu$ is \emph{regular} 
(see \cite[Definition 3.1.2]{StahlTotik1992}). In the formula \eqref{eq:97}, $\omega'$ denotes the density of the
\emph{equilibrium measure} corresponding to the support of $\mu$, see \eqref{eq:96} for details. In the case when
$\supp(\mu)$ is a finite union of compact intervals, $\mu$ is regular provided that $\mu' > 0$ almost everywhere
in the interior of $\supp(\mu)$. Let us recall that
\[
	\sinc(x) = 
	\begin{cases}
		\frac{\sin(x)}{x} & \text{if } x \neq 0, \\
		1 & \text{otherwise.}
	\end{cases}
\]

The best understood class of measures with unbounded support is the class of \emph{exponential weights}.
In the monograph \cite{Levin2001} asymptotics of the polynomials as well as their Jacobi parameters were studied under
a number of regularity conditions imposed on the function $Q(x) = -\log \mu'(x)$. Concerning the Christoffel--Darboux
kernel, it was recently proven in \cite{Ignjatovic2017} that under some conditions
\begin{equation} 
	\label{eq:108}
	\lim_{n \to \infty} \frac{1}{\rho_n} K_n(x,x) = \frac{1}{2 \pi \mu'(x)}
\end{equation}
locally uniformly with respect to $x \in \RR$ where
\[
	\rho_n = \sum_{j=0}^n \frac{1}{a_j}.
\]
Let us comment that $\rho_n$ is comparable to $n$, if the sequences $(a_n)$ and $(a_n^{-1})$ are bounded.
In particular, the conditions imposed on $Q$ imply that the density of $\mu$ is an even, everywhere positive and continuously differentiable function.

Instead of taking the measure $\mu$ as the starting point one can consider polynomials $(p_n : n \in \NN_0)$ satisfying
the three-term recurrence relation \eqref{eq:100} for a given sequences $(a_n)$ and $(b_n)$ such that $a_n > 0$ and
$b_n \in \RR$. In view of the Favard's theorem (see, e.g. \cite[Theorem 5.10]{Schmudgen2017}), there is a probability
measure $\nu$ such that $(p_n)$ is orthonormal in $L^2(\RR, \nu)$. The measure $\nu$ is unique, if and only if there
is exactly one measure with the same moments as $\nu$. In such a case we call $\nu$ \emph{determinate} and denote it by
$\mu$. Otherwise $\nu$ is \emph{indeterminate}. For example, the determinacy of $\nu$ is implied by the
\emph{Carleman condition}
\begin{equation}
	\label{eq:103}
	\sum_{n=0}^\infty \frac{1}{a_n} = \infty
\end{equation}
(see, e.g. \cite[Corollary 6.19]{Schmudgen2017}). However, the condition \eqref{eq:103} is not necessary. Let us recall
that the orthogonality measure has compact support, if and only if the Jacobi parameters are bounded.

In the setup when the Jacobi parameters are central objects, the questions concerning asymptotic behavior of orthogonal
polynomials and the Christoffel--Darboux kernel make a perfect sense. Additionally, if the measure $\nu$ is determinate,
one can ask how to approximate it.

In this article we are exclusively interested in \emph{unbounded} Jacobi parameters. We shall mostly consider
the class of periodically modulated sequences. This class has been introduced in \cite{JanasNaboko2002}
and systematically studied since then. To be more precise, let $N$ be a positive integer. We say that Jacobi parameters 
$(a_n)$ and $(b_n)$ are \emph{$N$-periodically modulated} if there are two $N$-periodic sequences $(\alpha_n : n \in \ZZ)$
and $(\beta_n : n \in \ZZ)$ of positive and real numbers, respectively, such that
\begin{enumerate}[(a)]
	\item
	$\begin{aligned}[b]
	\lim_{n \to \infty} a_n = \infty
	\end{aligned},$
	\item
	$\begin{aligned}[b]
	\lim_{n \to \infty} \bigg| \frac{a_{n-1}}{a_n} - \frac{\alpha_{n-1}}{\alpha_n} \bigg| = 0
	\end{aligned},$
	\item
	$\begin{aligned}[b]
	\lim_{n \to \infty} \bigg| \frac{b_n}{a_n} - \frac{\beta_n}{\alpha_n} \bigg| = 0
	\end{aligned}.$
\end{enumerate}
It turns out that properties of $\mu$ corresponding to $N$-periodically modulated Jacobi parameters are intimately related
to the trace of the matrix
\[
	\frakX_1(x) = \prod_{j=1}^N 
	\begin{pmatrix}
		0 & 1 \\
		-\frac{\alpha_{j-1}}{\alpha_j} & \frac{x-\beta_j}{\alpha_j}
	\end{pmatrix}.
\]
More precisely, under some regularity assumptions imposed on the Jacobi parameters, the measure $\mu$ is purely absolutely
continuous on $\RR$ with positive continuous density when $|\tr \frakX_1(0)| < 2$
(see \cite{PeriodicII, SwiderskiTrojan2019}), whereas $\mu$ is purely discrete when $|\tr \frakX_1(0)| > 2$ 
(see \cite{JanasNaboko2002}). In the boundary case $|\tr \frakX_1(0)| = 2$, we have two possibilities:
either the matrix $\frakX_1(0)$ is diagonalizable (which implies that $\frakX_1(0) = \sigma \Id$ for some
$\sigma \in \{-1, 1\}$), or it is similar to a non-trivial Jordan block. In this article we are concerned with the first
case. The second one is more challenging and we leave it for future research.

In the case $|\tr \frakX_1(0)| < 2$ the asymptotic behavior of $(p_n : n \in \NN_0)$ has already been studied in
\cite{SwiderskiTrojan2019}, and the scaling limit of $K_n$ has been recently obtained in \cite{ChristoffelI}.
In particular, it has been shown that under some regularity assumptions one has
\begin{equation} 
	\label{eq:109}
	\lim_{n \to \infty} \frac{1}{\rho_n} K_n \Big( x + \frac{u}{\rho_n}, x + \frac{v}{\rho_n} \Big) =
	\frac{\omega'(0)}{\mu'(x)} \sinc \big( (u-v) \pi \omega'(0) \big)
\end{equation}
locally uniformly with respect to $x,u,v \in \RR$, where
\begin{equation} \label{eq:110}
	\omega'(0) = \frac{|\tr \frakX_1'(0)|}{N \pi \sqrt{4 - \big( \tr \frakX_1(0) \big)^2}},
\end{equation}
and
\[
	\rho_n = \sum_{j = 0}^n \frac{\alpha_j}{a_j}.
\]
Notice that by taking $\alpha_n \equiv 1$, $\beta_n \equiv 0$, $N = 1$, and $u = v = 0$, we can reproduce \eqref{eq:108}.

Let us emphasize that $|\tr \frakX_1(0)| = 2$ when $\frakX_1(0) = \sigma \Id$. This situation lies on the boundary of the
previous case. In particular, the formula \eqref{eq:110} is not well-defined, and consequently, it is not clear whether
any analogue of \eqref{eq:109} holds true. Moreover, the behavior of the corresponding measure is different than in the
case $|\tr \frakX_1(0)| < 2$. In fact, under some regularity conditions, there is an explicit compact interval
$I \subset \RR$ such that the measure $\mu$ is purely absolutely continuous on $\RR \setminus I$ with continuous positive
density, and it is discrete on $I$ (see \cite{PeriodicIII}). Moreover, in our forthcoming article \cite{Discrete} we have
shown that in fact the support of $\mu$ has no accumulation points in the interior of $I$. All of this suggest that
the asymptotic behavior of the polynomials and the Christoffel--Darboux kernel might be different in this setup.
A very natural example satisfying our theorems are the following Jacobi parameters
\[
	a_n = (n+1)^{\kappa} + \frac{1}{2}\big( (-1)^n + 1\big), \qquad b_n \equiv 0,
\]
for $\kappa \in (0,1)$ and $N = 2$. In view of \cite{Dombrowski2004}, the measure $\mu$ is absolutely continuous on
$\RR$ and has support equal to $\RR \setminus (-1, 1)$.

Before we go further, let us introduce some terminology. A sequence $(u_n \in \NN_0)$ is \emph{generalized 
eigenvector} associated with $x \in \RR$, if it satisfies the recurrence relation
\begin{equation}
	\label{eq:28}
	x u_n = a_n u_{n+1} + b_n u_n + a_{n-1} u_{n-1}, \qquad n \geq 1,
\end{equation}
with some initial condition $(u_0, u_1) \neq (0, 0)$. The relation \eqref{eq:28} can be rewritten as
\begin{equation}
	\label{eq:136}
	\begin{pmatrix}
		u_n \\
		u_{n+1}
	\end{pmatrix}
	=
	B_n(x)
	\begin{pmatrix}
		u_{n-1} \\
		u_n
	\end{pmatrix}
\end{equation}
where $B_n$ is \emph{$1$-step transfer matrix} defined as
\[
	B_n(x) = 
	\begin{pmatrix}
		0 & 1 \\
		-\frac{a_{n-1}}{a_n} & 	\frac{x-b_n}{a_n}
	\end{pmatrix}.
\]
To study $N$-periodic modulations we consider $N$-step transfer matrix defined as
\[
	X_n = \prod_{j=n}^{n+N-1} B_j.
\]
By $\GL(2, \RR)$ we denote $2 \times 2$ 
real invertible matrices equipped with the spectral norm. For a matrix 
\[
	Y = 
	\begin{pmatrix}
	y_{1, 1} & y_{1, 2} \\
	y_{2, 1} & y_{2, 2}
	\end{pmatrix}
\]
we set $[Y]_{i, j} = y_{i, j}$; its discriminant is defined as $\discr Y = (\tr Y)^2 - 4 \det Y$.
Given a compact subset $K \subset \RR$, we say that the sequence $(Y_n : n \in \NN)$ of
mappings $Y_n : K \to \GL(2, \RR)$ belongs to $\calD_1 \big(K, \GL(2, \RR) \big)$, if
\[
	\sum_{n=0}^\infty \sup_{x \in K} \| Y_{n+1}(x) - Y_n(x) \| < \infty.
\]

It turns out that some properties of the measure $\mu$ depend on the asymptotic behavior of generalized eigenvectors.
For example, $\mu$ is determinate, if and only if there is a generalized eigenvector corresponding to $x\in \RR$ which
is not square-summable. Moreover, subordinacy theory (see e.g. \cite{Khan1992}) implies that, if $\mu$ is determinate
and $I \subset \RR$ is an open interval such that for any generalized eigenvectors $(u_n)$, $(v_n)$ associated with
$x \in I$,
\[
	\sup_{n \geq 0} \frac{\sum_{k=0}^n |u_k|^2}{\sum_{k=0}^n |v_k|^2} < \infty,
\]
then $\mu$ is absolutely continuous on $I$, and $I \subset \supp(\mu)$. This motivates the study of the asymptotic
behavior of generalized eigenvectors.
\begin{main_theorem}
	\label{thm:A}
	Let $N$ be a positive integer, and $i \in \{0, 1, \ldots, N-1 \}$. Suppose that
	\[
		\lim_{j \to \infty} a_{jN+i-1} = \infty.
	\]
	Let
	\[
		\Lambda^i = \Big\{ x \in \RR : \lim_{j \to \infty} \discr R_{jN+i}(x) \text{ exists and is negative} \Big\}
	\]
	where
	\[
		R_n(x) = a_{n+N-1} \big( X_n(x) - \sigma \Id \big)
	\]
	for some $\sigma \in \{-1, 1\}$. Suppose that $K \subset \Lambda^i$ is a compact interval with non-empty
	interior such that
	\[
		\big(X_{jN+i} : j \in \NN\big), \big(R_{jN+i} : j \in \NN\big) \in \calD_1 \big( K, \GL(2, \RR) \big).
	\]
	Then there is a constant $c>1$ such that for every generalized eigenvector $(u_n : n \in \NN_0)$ associated with
	$x \in K$, and all $n \geq 1$,
	\begin{equation} 
		\label{thmA:eq:1}
		c^{-1} \big( u_0^2 + u_1^2 \big) 
		\leq 
		a_{nN+i-1} \big( u_{nN+i-1}^2 + u_{nN+i}^2 \big)
		\leq
		c \big( u_0^2 + u_1^2 \big).
	\end{equation}
\end{main_theorem}
Theorem~\ref{thm:A} is a generalization of \cite[Theorem C]{PeriodicIII}, and is proven in Section \ref{sec:2}, see
Theorem \ref{thm:3}. If the hypotheses of Theorem \ref{thm:A} are satisfied for all $i \in \{0, 1, \ldots, N-1\}$,
the Carleman condition \eqref{eq:103}, with a help of subordinacy theory, implies that the bounds \eqref{thmA:eq:1} entail
the absolute continuity of the measure $\mu$ on each compact subset $K \subset \bigcap_{i=0}^{N-1} \Lambda^i$. However,
this method does not give any additional information about $\mu$. Because of this and inspired by earlier
work for bounded Jacobi parameters \cite{GeronimoVanAssche1991, Nevai1983}, in \cite{PeriodicII} the first author has
introduced a method in the case of unbounded Jacobi parameters that allows to approximate $\mu$ in terms of $N$-shifted
Tur\'an determinants. The later are defined as follows
\begin{align*}
	D^N_n(x) &= 
	\det
	\begin{pmatrix}
		p_{n+N-1}(x) & p_{n-1}(x) \\
		p_{n+N}(x) & p_n(x)
	\end{pmatrix} \\
	&=
	p_n(x) p_{n+N-1}(x) - p_{n-1}(x) p_{n+N}(x).
\end{align*}
See also \cite{PeriodicIII, SwiderskiTrojan2019} and the references therein.

Our second result concerns the convergence of $N$-shifted Tur\'an determinants.
\begin{main_theorem} 
	\label{thm:B}
	Suppose that the hypotheses of Theorem~\ref{thm:A} are satisfied. Assume further that
	\[
		\lim_{j \to \infty} \big( a_{(j+1)N+i-1} - a_{jN+i-1} \big) = 0.
	\]
	Then there is a positive function $g_i$, such that
	\begin{equation}
		\label{thmB:eq:1}
		\lim_{j \to \infty}
		\sup_{x \in K}
		\Big|
		a_{(j+1)N+i-1}^2 \big| D^N_{jN+i}(x) \big| - g_i(x)
		\Big|
		=0.
	\end{equation}
	Moreover, the measure $\mu$ is absolutely continuous on $K$ with the density
	\[
		\mu'(x) = \frac{\sqrt{-h_i(x)}}{2\pi g_i(x)}, \qquad x \in K
	\]
	where
	\begin{equation}
		\label{eq:12}
		h_i(x) = \lim_{n \to \infty} \discr \big( R_{nN+i}(x) \big), \qquad x \in K.
	\end{equation}
\end{main_theorem}
Again, Theorem~\ref{thm:B} is a generalization of \cite[Theorem D]{PeriodicIII}. Its proof is in Section \ref{sec:3},
see Theorem \ref{thm:4}. 

In the next theorem we study asymptotics of the polynomials in more detail.
\begin{main_theorem}
	\label{thm:C}
	Suppose that the hypotheses of Theorem~\ref{thm:B} are satisfied. Then there are $M \geq 1$ and a continuous
	real-valued function $\eta$ such that for all $n \geq M$,
	\begin{equation}
		\label{thm:C:eq:1}
		\lim_{n \to \infty} \sup_{x \in K} 
		\bigg|
	 	\sqrt{a_{nN+i-1}} p_{nN+i}(x) -
	 	\sqrt{\frac{2 \big|[\calR_i(x)]_{2,1}\big|}{\pi \mu'(x) \sqrt{-h_i(x)}}}
	 	\sin \Big( \sum_{j=M+1}^{n} \theta_j(x) + \eta(x) \Big)
		\bigg| = 0
	\end{equation}
	where $h_i$ is given by \eqref{eq:12},
	\[
		\calR_i(x) = \lim_{n \to \infty} R_{nN+i}(x),
	\]
	and
	\[
		\theta_j(x) = 
		\arccos \bigg( \frac{\tr X_{jN+i}(x)}{2 \sqrt{\det X_{jN+i}(x)}} \bigg).
	\]
\end{main_theorem}
Let us mention that asymptotics of the polynomials in the case when $(X_{nN+i} : n \in \NN)$ is convergent to the matrix
$\calX$ such that $\discr \calX < 0$, has been obtained in \cite{SwiderskiTrojan2019}. Theorem~\ref{thm:C} corresponds
to the case when $\discr \calX = 0$. For the proof of Theorem \ref{thm:C} we refer to Theorem \ref{thm:1}.

Finally, in the last theorem we study the analogues of \eqref{eq:109} for periodic modulations such that 
$\frakX_1(0) = \sigma \Id$ for some $\sigma	\in \{-1, 1 \}$. It is proven in Theorem \ref{thm:7}.
\begin{main_theorem} 
	\label{thm:D}
	Suppose that the hypotheses of Theorem~\ref{thm:B} are satisfied for all $i \in \{0, 1, \ldots, N-1 \}$. Assume
	further that Jacobi parameters $(a_n)$ and $(b_n)$ are $N$-periodically modulated so that $\frakX_1(0) = \sigma \Id$.
	Then
	\begin{equation}
		\label{thm:D:eq:1}
		\lim_{n \to \infty}
		\frac{1}{\rho_n} K_n\bigg(x + \frac{u}{\rho_n}, x + \frac{v}{\rho_n} \bigg)
		=
		\frac{\upsilon(x)}{\mu'(x)}
		\sinc \big( (u-v) \pi \upsilon(x) \big)
	\end{equation}
	locally uniformly with respect to $(x,u,v) \in \Lambda^0 \times \RR^2$, where
	\[
		\rho_n = \sum_{j=0}^n \frac{\alpha_j}{a_j}, 
		\qquad\text{and}\qquad
		\upsilon(x) = \frac{1}{4 \pi N \alpha_{N-1}} \frac{|h_0'(x)|}{\sqrt{-h_0(x)}}.
	\]
\end{main_theorem}
Observe that in \eqref{eq:109} the factor $\omega'(0)$ is constant whereas in \eqref{thm:D:eq:1}
the factor $\upsilon$ usually depends on $x \in \Lambda^0$. For cases when $\upsilon(x) = \omega'(0)$ for each $x \neq 0$
see Theorem~\ref{thm:2}. By taking $u = v = 0$ in Theorem~\ref{thm:D} we obtain sufficient conditions under which
Ignjatović's conjecture is valid, see Corollary~\ref{cor:2} for details. In general the conjecture is false.

Let us present some ideas of the proofs. In view of \cite[Theorem 1]{PeriodicI} it is enough to prove the uniform
convergence of generalized $N$-shifted Tur\'an determinants
\[
	a_{n+N-1}^2 
	\det
	\begin{pmatrix}
		u_{n+N-1} & u_{n-1} \\
		u_{n+N} & u_n
	\end{pmatrix}.
\]
We do so by careful analysis of $(X_{nN+i} : n \in \NN)$ following the method developed in \cite{PeriodicIII},
see Section~\ref{sec:2} for details. To prove Theorem \ref{thm:B}, we use the convergence of generalized $N$-shifted
Tur\'an determinants together with the approximation method described in \cite{PeriodicII}, see Section \ref{sec:3}.
To prove Theorem~\ref{thm:C} we follow the method recently introduced in \cite{SwiderskiTrojan2019}
for the case when the limit of the sequence $(X_{nN+i} : n \in \NN)$ has negative discriminant. In the current
setup the analysis is much more subtle and involved. Finally, to prove Theorem~\ref{thm:D} we use the asymptotics from
Theorem~\ref{thm:C}. It results in estimating the following oscillatory sum
\[
	\sum_{k=0}^n \frac{\gamma_k}{\sum_{j=0}^n \gamma_j}
	\sin \Big( \sum_{j=0}^n \theta_j(x_n) + \sigma(x_n) \Big)
	\sin \Big( \sum_{j=0}^n \theta_j(y_n) + \sigma(y_n) \Big)
\]
where 
\[
	x_n = x + \frac{u}{\rho_n}, \qquad\text{and}\qquad
	y_n = x + \frac{v}{\rho_n}.
\]
To deal with the sum we prove two auxiliary results (see Lemma \ref{lem:1} and Lemma \ref{lem:3}) that are 
valid for sequences not necessarily belonging to $\calD_1$. In Section~\ref{sec:6}, we show a number of necessary
algebraic identities which are specific to the periodic modulations in the setup $\frakX_1(0) = \sigma \Id$. 

The article is organized as follows: In Section \ref{sec:1} we fix some basic notation. Section \ref{sec:2} is devoted to
proving Theorem \ref{thm:A} (see Theorem \ref{thm:3}). In the next section we describe the approximation procedure 
(see Proposition \ref{prop:1}) which is a tool in proving Theorem \ref{thm:B} (see Theorem \ref{thm:4}). In Section
\ref{sec:4}, we study the asymptotic behavior of orthogonal polynomials. Behavior of the Christoffel function
in residue classes is analyzed in Section \ref{sec:5}. In the next section we define periodic modulations 
and introduce a function $\upsilon$. Section \ref{sec:7} is dedicated to study the Christoffel function for
periodic modulations. Finally, in Section \ref{sec:8} we investigate asymptotic behavior of the Christoffel--Darboux
kernel. Let us emphasize that Lemma \ref{lem:1} and Lemma \ref{lem:3} are sufficiently general to allow studying other 
types of scaling limits of Christoffel--Darboux kernels. In particular, in Theorem \ref{thm:11} it is shown how Lemma
\ref{lem:3} may be applied.

\subsection*{Notation}
By $\NN$ we denote the set of positive integers and $\NN_0 = \NN \cup \{0\}$. Throughout the whole article, we write $A \lesssim B$ if there is an absolute constant $c>0$ such that
$A\le cB$. Moreover, $c$ stands for a positive constant whose value may vary from occurrence to occurrence.

\subsection*{Acknowledgment}
The first author was partially supported by the Foundation for Polish Science (FNP) and by long term structural funding -- Methusalem grant of the Flemish Government.

\section{Preliminaries}
\label{sec:1}
Given Jacobi parameters $(a_n : n \in \NN_0)$ and $(b_n : n \in \NN_0)$ and $k \in \NN_0$, we define polynomials 
$(p_n^{[k]} : n \in \NN_0)$ by relations
\begin{align*}
	p_0^{[k]}(x) &= 1, \qquad p_1^{[k]}(x) = \frac{x-b_k}{a_k}, \\
	x p_n^{[k]}(x) &= a_{n+k-1} p_{n-1}^{[k]}(x) + b_{n+k} p_n^{[k]}(x) + a_{n+k} p_{n+1}^{[k]}(x), \qquad n \geq 1.
\end{align*}
We usually omit the superscript when $k = 0$. In particular, $(p_n(x) : n \in \NN_0)$ are the generalized eigenvectors
associated with $x$ satisfying the initial condition
\[
	p_0(x) = 1, \qquad p_1(x) = \frac{x-b_0}{a_0}.
\]
Given a compact set $K \subset \RR$ with non-empty interior, there is the unique probability measure $\omega_K$,
called the \emph{equilibrium measure} corresponding to $K$, minimizing the energy
\begin{equation} \label{eq:96}
	I(\nu) = -\int_\RR \int_\RR \log |x-y| \nu({\rm d} x) \nu({\rm d} y)
\end{equation}
among all probability measures $\nu$ supported on $K$. The measure $\omega_K$ is absolutely continuous on the interior
of $K$ with continuous density, see \cite[Theorem IV.2.5, pp. 216]{Saff1997}.

Let $r$ be a positive integer. We say that a sequence $(x_n : n \in \NN)$ of vectors from a normed space $V$ belongs
to $\calD_r(V)$, if it is bounded and for each $j \in \{1, \ldots, r\}$,
\[
	\sum_{n = 1}^\infty \big\| \Delta^j x_n \big\|^{\frac{r}{j}} < \infty
\]
where
\begin{align*}
	\Delta^0 x_n &= x_n, \\
	\Delta^j x_n &= \Delta^{j-1} x_n - \Delta^{j-1} x_{n-1}, \qquad j \geq 1.
\end{align*}
For a positive integer $N$, we say that
\[
	(x_n : n \in \NN) \in \calD_r^N(V)
\]
if for each $i \in \{0, 1, \ldots, N-1 \}$
\[
	(x_{nN+i} : n \in \NN ) \in \calD_r(V).
\]
If $Y$ is a matrix
\[
	Y = \begin{pmatrix}
		y_{1,1} & y_{1,2} \\
		y_{2,1} & y_{2,2}
	\end{pmatrix},
\]
we set $[Y]_{i, j} = y_{i,j}$. The symmetrization and the discriminant of $Y$ are defined as
\[
	\sym(Y) = \frac{1}{2} Y + \frac{1}{2} Y^*,
	\qquad\text{and}\qquad
	\discr Y = (\tr Y)^2 - 4 \det Y,
\]
respectively. Here $Y^*$ is the Hermitian transpose of the matrix $Y$.

\section{Tur\'an determinants}
\label{sec:2}
Let $N$ be a positive integer and let $(u_n : n \in \NN_0)$ be a generalized eigenvector associated with
$\alpha \in \RR^2 \setminus \{ 0 \}$ and $x \in \RR$. We define $N$-shifted generalized Tur\'an determinant by the
formula
\begin{equation}
	\label{eq:40}
	D_n(\alpha, x) = u_n(x) u_{n+N-1}(x) - u_{n-1}(x) u_{n+N}(x).
\end{equation}

\begin{theorem} 
	\label{thm:3}
	Let $N$ be a positive integer and $i \in \{0, 1, \ldots, N-1 \}$. Suppose that
	\[
		\lim_{j \to \infty} a_{jN+i-1} = \infty.
	\]
	Let
	\[
		\Lambda = \big\{ x \in \RR : \lim_{j \to \infty} \discr R_{jN+i}(x) \text{ exists and is negative} \big\}
	\]
	where
	\[
		R_n(x) = a_{n+N-1} ( X_n(x) - \sigma \Id )
	\]
	for some $\sigma \in \{-1, 1\}$. Suppose that $K \subset \Lambda$ is a compact interval with non-empty interior and
	$\Omega$ is a compact and connected subset of $\RR^2 \setminus \{0 \}$. If
	\[
		\big(X_{jN+i} : j \in \NN\big), \big(R_{jN+i} : j \in \NN\big) \in \calD_1 \big( K; \GL(2, \RR) \big),
	\]
	then there is $g$ a real continuous function without zeros on $\Omega \times K$ and a constant $c>0$ such that
	\begin{align*}
	&\sup_{\alpha \in \Omega} \sup_{x \in K}
	\big| a_{(k+1)N+i-1}^2 D_{kN+i}(\alpha, x) - g(\alpha, x) \big| \\
	&\qquad
	\leq c 
	\sum_{j=k}^\infty \Big( \sup_{x \in K} \big\| R_{(j+1)N+i}(x) - R_{jN+i}(x) \big\| 
	+ \sup_{x \in K} \big\| X_{(j+1)N+i}(x) - X_{jN+i}(x) \big\| \Big).
\end{align*}
\end{theorem}
\begin{proof}
	We follow the method developed in \cite[Theorem 2]{PeriodicI}, namely we write
	\[
		\frac{S_{(k+1)N+i-1}}{S_{m N+i-1}} 
		= \prod_{j = m}^k (1 + F_{jN+i-1})
	\]
	where for $x \in \RR$ and $\alpha \in \RR^2 \setminus \{0\}$ we have set
	\[
		S_n(\alpha, x) = a_{n+N-1}^2 D_n(\alpha, x),
	\]
	and
	\[
		F_{jN+i-1} = \frac{S_{(j+1)N+i-1} - S_{jN+i-1}}{S_{jN+i-1}}.
	\]
	Therefore, for the proof it is sufficient to show
	\begin{equation}
		\label{eq:15}
		\sum_{j = m}^\infty \sup_{\alpha \in \Omega} \sup_{x \in K} |F_{jN+i-1}(\alpha, x) | < \infty.
	\end{equation}
	In view of \cite[Lemma 3]{PeriodicIII} we have
	\begin{equation}
		\label{eq:29}
		\begin{aligned}
		&
		\big|
		S_{n+N}(\alpha, x) - S_n(\alpha, x)
		\big| \\
		&\qquad\qquad
		\leq
		a_{n+N-1}
		\bigg\|
		\frac{a_{n+2N-1}}{a_{n+N-1}} R_{n+N}(x) - 
		\frac{a_{n+N-1}}{a_{n-1}} R_n(x) \bigg\|
		\big(u_{n+N-1}(x)^2 + u_{n+N}(x)^2\big).
		\end{aligned}
	\end{equation}
	We next consider a quadratic form on $\RR^2$ defined as
	\[
		Q^x_n(v) = a_{n+N-1} \big\langle E X_n(x) v, v \big\rangle
	\]
	where $x \in \RR$, and
	\[
		E = 
		\begin{pmatrix}
		0 & -1 \\
		1 & 0
		\end{pmatrix}.
	\]
	In particular, we have
	\begin{equation}
		\label{eq:34}
		S_n (\alpha, x) = a_{n+N-1} Q^x_n
		\begin{pmatrix}
		u_{n-1} \\
		u_{n}
		\end{pmatrix}.
	\end{equation}
	We claim the following holds true.
	\begin{claim}
		\label{clm:1}
		There is $c > 0$ such that for all $j \in \NN$, $x \in K$, and $v \in \RR^2$,
		\begin{equation}
			\label{eq:9}
			c^{-1} (v_1^2 + v_2^2) \leq \big|Q^x_{jN+i}(v) \big| \leq c (v_1^2 + v_2^2).
		\end{equation}
	\end{claim}
	For the proof, let us observe that
	\[
		Q_{jN+i}^x(v) = \frac{a_{(j+1)N+i-1}}{a_{jN+i-1}} \big\langle E R_{jN+i}(x) v, v\big\rangle.
	\]
	Since the sequence $(R_{jN+i} : j \in \NN_0)$ belongs to $\calD_1\big(K, \GL(2, \RR)\big)$, it is convergent. 
	Let $\calR_i(x)$ denote its limit. We next notice that
	\begin{align*}
		\lim_{j \to \infty} \det X_{jN+i}(x)
		&=
		\lim_{j \to \infty} \det \big( a_{jN+i-1}^{-1} R_{jN+i} + \sigma \Id\big) \\
		&=
		\sigma^2 = 1.
	\end{align*}
	On the other hand,
	\[
		\det X_{jN+i}(x) = \frac{a_{jN+i-1}}{a_{(j+1)N+i-1}}.
	\]
	Consequently, we obtain
	\[
		\lim_{j \to \infty} \frac{a_{(j+1)N+i-1}}{a_{jN+i-1}} = 1,
	\]
	and hence
	\[
		\lim_{j \to \infty} Q_{jN+i}^x(v) = \big\langle E \calR_i(x) v, v \big\rangle.
	\]
	Since
	\[
		\det \sym(E \calR_i) = -\tfrac{1}{4} \discr \calR_i,
	\]
	we conclude that for each $x \in K$,
	\[
		\det \sym(E \calR_i(x)) > 0,
	\]
	which easily leads to \eqref{eq:9}.

	Using now Claim \ref{clm:1} together with \eqref{eq:29} and \eqref{eq:34}, we obtain
	\begin{equation}
		\label{eq:35}
		\sup_{\alpha \in \Omega} \sup_{x \in K} 
		\big|
		F_{jN+i-1}(\alpha, x)
		\big|
		\leq
		c 
		\sup_{x \in K}
		\bigg\|
		\frac{a_{(j+2)N+i-1}}{a_{(j+1)N+i-1}} R_{(j+1)N+i}(x)
		-
		\frac{a_{(j+1)N+i-1}}{a_{jN+i-1}} R_{jN+i}(x)
		\bigg\|.
	\end{equation}
	By \cite[the formula (6)]{PeriodicI}, there is a constant $c > 0$, such that
	\begin{align*}
		&\sum_{j=m}^\infty \sup_K
		\bigg\|
		\frac{a_{(j+2)N + i-1}}{a_{(j+1)N+i-1}} R_{(j+1)N+i} 
		- 
		\frac{a_{(j+1)N + i-1}}{a_{jN+i-1}} R_{jN+i} 
		\bigg\| \\	
		&\qquad\qquad\leq 
		c \sum_{j=m}^\infty
		\bigg( \sup_K \big\| R_{(j+1)N+i} - R_{jN+i} \big\| + 
		\bigg| 
		\frac{a_{(j+2)N + i-1}}{a_{(j+1)N+i-1}}
		-
		\frac{a_{(j+1)N + i-1}}{a_{jN+i-1}}
		\bigg|
		\bigg).
	\end{align*}
	Since for each $r, s \in \{1, 2\}$,
	\[
		\Big| \big[ X_{(j+1)N+i}(x) \big]_{r, s} - \big[ X_{jN+i}(x) \big]_{r, s} \Big|
		\leq
		\big\| X_{(j+1)N+i}(x) - X_{jN+i}(x) \big\|,
	\]
	we obtain
	\[
		\sum_{j = m}^\infty \sup_K \big| \det X_{(j+1)N+i} - \det X_{jN+i} \big| \leq c
		\sum_{j = m}^\infty \sup_K \big \| X_{(j+1)N+i} - X_{jN+i} \big\|.
	\]
	Hence,
	\begin{align*}
		&\sum_{j = m}^\infty \sup_K
		\bigg\| 
		\frac{a_{(j+2)N + i-1}}{a_{(j+1)N+i-1}} R_{(j+1)N+i} 
		- \frac{a_{(j+1)N + i-1}}{a_{jN+i-1}} R_{jN+i}
		\bigg\| \\
		&\qquad\qquad
		\leq 
		c \sum_{j=n}^\infty
		\Big( \sup_K \big\| R_{(j+1)N+i} - R_{jN+i} \big\| +
		\sup_K \big\| X_{(j+1)N+i} - X_{jN+i} \big\| \Big),
	\end{align*}
	which together with \eqref{eq:35} implies \eqref{eq:15} and the theorem follows.
\end{proof}

\section{Approximation procedure}
\label{sec:3}
In this section we present a method that allows to prove a formula for the density of an orthogonality measure.
It is a further development of \cite{PeriodicII} and \cite{SwiderskiTrojan2019}.

Let $(p_n : n \in \NN_0)$ be a sequence of polynomials corresponding to sequences $(a_n : n \in \NN_0)$ and
$(b_n : n \in \NN_0)$. We set
\begin{equation} 
	\label{eq:135}
	\scrD_n(x) = p_n(x) p_{n+N-1}(x) - p_{n-1}(x) p_{n+N}(x).
\end{equation}
Then
\[
	\scrD_n(x)
	=
	D_n \Big( \Big(1, \tfrac{x - b_0}{a_0} \Big), x \Big)
\]
where $D_n$ is given by \eqref{eq:40}. 

Given $L \in \NN$, we consider the truncated sequences $(a^L_n : n \in \NN_0)$ and $(b^L_n : n \in \NN_0)$ defined by
\begin{subequations}
	\begin{equation}
	\label{eq:42a}
	a^L_n = 
	\begin{cases}
		a_n & \text{if } 0 \leq n < L+N, \\
		a_{L+i} & \text{if } L+N \leq n, \text{ and } n-L \equiv i \bmod N,
	\end{cases}
	\end{equation}
	and
	\begin{equation}
	\label{eq:42b}
	b^L_n =
	\begin{cases}
		b_n & \text{if } 0 \leq n < L+N, \\
		b_{L+i} & \text{if } L+N \leq n, \text{ and } n-L \equiv i \bmod N,
	\end{cases}
	\end{equation}
\end{subequations}
where $i \in \{0, 1, \ldots, N-1\}$. Let $(\scrD_n^L : n \in \NN_0)$ be the sequence \eqref{eq:135} associated to the
polynomials $(p_n^L : n \in \NN_0)$ that are corresponding to the sequences $a^L$ and $b^L$. Then
\[
	\scrD_n^L(x) = 
	\bigg\langle
	E X_n^L(x)
	\begin{pmatrix}
		p^L_{n-1}(x) \\
		p^L_{n}(x)
	\end{pmatrix},
	\begin{pmatrix}
		p^L_{n-1}(x) \\
		p^L_{n}(x)
	\end{pmatrix}
	\bigg\rangle
\]
where
\begin{equation} 
	\label{eq:43}
	X_n^L(x) = \prod_{j=n}^{n+N-1}
	\begin{pmatrix}
		0 & 1\\
		-\frac{a_{j-1}^L}{a_j^L} & \frac{x - b_j^L}{a_j^L}
	\end{pmatrix} \quad \text{and} \quad
	E = 
	\begin{pmatrix}
		0 & -1 \\
		1 & 0	
	\end{pmatrix}.
\end{equation}
Observe that there is the unique measure $\mu_L$ orthonormalizing the polynomials $(p^L_n : n \in \NN_0)$.

\begin{proposition}
	\label{prop:1}
	Given $\sigma \in \{-1, 1\}$ we set
	\begin{equation}
		\label{prop:1:eq:1}
		R_n = a_{n+N-1} \big( X_n - \sigma \Id \big), \qquad \text{and} \qquad
		R_n^L = a^L_{n+N-1} \big( X^L_n - \sigma \Id \big).
	\end{equation}
	Suppose that $(L_j : j \in \NN)$ is an increasing sequence of integers such that
	\begin{equation}
		\label{prop:1:eq:2}
		\lim_{j \to \infty} a_{L_j-1} = \infty, \qquad\text{and}\qquad
		\lim_{j \to \infty} (a_{L_j + N-1} - a_{L_j - 1}) = 0.
	\end{equation}
	If $K$ is a compact subset of $\RR$, such that
	\begin{equation} 
		\label{prop:1:eq:3}
		\sup_{j \in \NN} \sup_{K} \| R_{L_j} \| < \infty,
	\end{equation}
	then
	\[
		\lim_{j \to \infty} 
		\sup_K 
		\left\| R_{L_j} - R^{L_j}_{L_j+N} \right\| 
		= 0.
	\]
\end{proposition}
\begin{proof}
	Since
	\[
		R_L - R_{L+N}^L = a_{L+N-1} \big( X_L - X_{L+N}^L \big),
	\]
	in view of \cite[Corollary 4]{SwiderskiTrojan2019} we have
	\begin{equation}
		\label{eq:20}
		\big\| R_L - R_{L+N}^L \big\| 
		\leq
		\| X_L \| 
		\frac{a_{L+N-1}}{a_{L-1}} \big| a_{L+N-1} - a_{L-1} \big|.
	\end{equation}
	By \eqref{prop:1:eq:1}, we have
	\[
		\| X_n \| \leq 1 + \frac{\|R_n\|}{a_{n+N-1}},
	\]
	thus, \eqref{prop:1:eq:2} and \eqref{prop:1:eq:3} imply that
	\begin{equation}
		\label{eq:18}
		\sup_{j \in \NN} \sup_K \| X_{L_j} \| < \infty.
	\end{equation}
	Lastly, by \eqref{prop:1:eq:2} we have
	\begin{equation}
		\label{eq:19}
		\lim_{j \to \infty} \frac{a_{L_j+N-1}}{a_{L_j-1}} = 1.
	\end{equation}
	Therefore, by using \eqref{eq:18}, \eqref{eq:19} and \eqref{prop:1:eq:2} in \eqref{eq:20}, the conclusion follows.
\end{proof}

\begin{theorem} \label{thm:4}
	Let $N$ be a positive integer and $\sigma \in \{-1, 1\}$. We set
	\[
		R_n = a_{n+N-1} \big(X_n - \sigma \Id \big).
	\]
	Suppose $(L_j : \in \NN)$ is an increasing sequence of positive integers such that
	\begin{equation} 
		\label{eq:50a}
		\lim_{j \to \infty} a_{L_j - 1} = \infty,
	\end{equation}
	and
	\begin{equation} 
		\label{eq:50b}
		\lim_{j \to \infty} \big( a_{L_j + N-1} - a_{L_j - 1} \big) = 0.
	\end{equation}
	Let $K$ be a compact subset of 
	\[
		\Lambda = \Big\{x \in \RR : \lim_{j \to \infty} \discr R_{L_j}(x)
		\text{ exists and is negative} \Big\}
	\]
	with non-empty interior and such that
	\begin{equation}
		\label{eq:52}
		\sup_{j \in \NN} \sup_{x \in K} \| R_{L_j}(x) \| < \infty.
	\end{equation}
	Suppose that there is a positive function $g: K \rightarrow \RR$ such that 
	\begin{equation} 
		\label{eq:51}
		\lim_{j \to \infty} \sup_{x \in K} 
		\Big| a_{L_j+N-1}^2 \big| \scrD_{L_j}(x) \big| - g(x) \Big| = 0.
	\end{equation}
	If $\nu$ is any weak accumulation point of the sequence $(\mu_{L_j} : j \in \NN)$, then 
	$\nu$ is a probability measure such that $(p_n : n \in \NN_0)$ are orthogonal in 
	$L^2(\RR, \nu)$, which is absolutely continuous on $K$ with the density
	\[
		\nu'(x) = \frac{\sqrt{-h(x)}}{2 \pi g(x)}, \qquad x \in K
	\]
	where
	\begin{equation} \label{eq:17}
		h(x) = \lim_{j \to \infty} \discr R_{L_j}(x), \qquad x \in K.
	\end{equation}
\end{theorem}
\begin{proof}
	Let us fix a positive integer $L$. We set
	\begin{equation} 
		\label{eq:14}
		R_{n}^{L} = a^{L}_{n+N-1} \big( X_{n}^{L} - \sigma \Id \big),
	\end{equation}
	and
	\[
		\Lambda_L = \Big\{ x \in \RR: \discr\big(R^L_{L+N}(x)\big) < 0 \Big\},
	\]
	and
	\[
		S^L_n(x) = \big( a^L_{n+N-1} \big)^2 \scrD^L_{n}(x), \qquad n \geq 1.
	\]
	By \eqref{eq:14},
	\[
		\discr\big(X^L_{L+N}(x)\big) = \frac{1}{a_{L+N-1}^2}\discr\big(R^L_{L+N}(x)\big).
	\]
	Hence, by \cite[Theorem 3]{PeriodicII}, (see also \cite[Theorem 6]{GeronimoVanAssche1991}), 
	for each $x \in \Lambda_L$,
	\[
		\text{the limit}\quad 
		\lim_{k \to \infty} \big|S^L_{L+kN}(x)\big| 
		\quad\text{exists,}
	\]
	and defines a positive continuous function $g^L: \Lambda_L \rightarrow \RR$. Moreover, 
	the orthogonality measure $\mu_L$ is absolutely continuous on $\Lambda_L$ with the density
	\begin{align}
		\nonumber
		\mu'_L(x) &= \frac{a_{L+N-1} \sqrt{-\discr\big(X^L_{L+N}(x)\big)}}{2 \pi g^L(x)} \\
		\label{eq:55}
		&= 
		\frac{\sqrt{-\discr\big(R^L_{L+N}(x)\big)}}{2 \pi g^L(x)}.
	\end{align}
	Next, we observe that by estimates (21) and (22) from \cite{PeriodicI}
	\begin{align*}
		&\Big| S^L_{n+N}(x) - S^L_{n}(x) \Big| \\
		&\qquad
		\leq 
		a_{n+N-1}^L
		\Big( \big( p^L_{n+N-1}(x) \big)^2 + \big( p^L_{n+N-1}(x) \big)^2 \Big)
		\bigg\|
			\frac{a^L_{n+2N-1}}{a^L_{n+N-1}} R^L_{n+N}(x) -
			\frac{a^L_{n+N-1}}{a^L_{n-1}} R^L_{n}(x)
		\bigg\|,
	\end{align*}
	which together with \eqref{eq:42a} and \eqref{eq:42b} entails that $S^L_{n+N}(x) = S^L_n(x)$ for all 
	$n \geq L+1$. Hence, for all $x \in \Lambda_L$,
	\begin{equation}
		\label{eq:11}
		g^L(x) = \big| S^L_{L+N}(x) \big|.
	\end{equation}
	Let us fix a compact subset $K \subset \Lambda$. Since $\discr \big( R_{L_j}(x) \big)$ is 
	a polynomial of degree at most $2N$, the convergence in \eqref{eq:17} is uniform on $K$. 
	Thus, by Proposition~\ref{prop:1},
	\begin{equation}
		\label{eq:39}
		\lim_{j \to \infty}
		\sup_{x \in K}
		\Big| \discr \Big( R_{L_j+N}^{L_j}(x) \Big) - h(x) \Big| = 0.
	\end{equation}
	Moreover, $K \subset \Lambda_{L_j}$ for all $j$ sufficiently large. Now, setting
	\[
		S_n(x) = a_{n+N-1}^2 \scrD_n(x),
	\]
	by \cite[Proposition 5]{SwiderskiTrojan2019}, we obtain
	\begin{align*}
		\big| S^{L_j}_{L_j+N}(x) - S_{L_j}(x) \big| 
		&= 
		a_{L_j+N-1}^2 \big|\scrD^{L_j}_{L_j+N}(x) - \scrD_{L_j}(x)\big| \\
		&\leq
		a_{L_j+N-1} \big( p_{L_j+N-1}^2(x) + p_{L_j+N}^2(x) \big)
		\| X_{L_j}(x) \| \left|a_{L_j+N-1} - a_{L_j-1} \right|.
	\end{align*}
	From \cite[Lemma 3]{PeriodicIII} we have
	\[
		S_L(x) = \frac{a^2_{L+N-1}}{a_{L-1}}
		\left\langle
			E R_L(x)
			\begin{pmatrix}
				p_{L+N-1}(x) \\
				p_{L+N}(x)
			\end{pmatrix},
			\begin{pmatrix}
				p_{L+N-1}(x) \\
				p_{L+N}(x)
			\end{pmatrix}
		\right\rangle.
	\]	
	Since $K$ is a compact subset of $\Lambda$, there are $j_0$ and $\delta > 0$ such that for all $j \geq j_0$ and 
	$x \in K$ we have
	\begin{equation}
		\label{eq:37}
		\det \Big( \sym \big( E R_{L_j}(x) \big) \Big) = - \frac{1}{4} \discr \big( R_{L_j}(x) \big) \geq \delta.
	\end{equation}
	By \eqref{eq:50a} and \eqref{eq:50b}
	\[
		\lim_{j \to \infty} \frac{a_{L_j+N-1}}{a_{L_j-1}} = 1,
	\]
	which together with \eqref{eq:37} implies that there are $j_0$ and $c > 0$ such that for all $j \geq j_0$ and
	$x \in K$,
	\[
		\big| S_{L_j}(x) \big| \geq c^{-1} a_{L_j+N-1} \big( p_{L_j+N-1}^2(x) + p_{L_j+N}^2(x) \big).
	\]
	Hence,
	\[
		\big|S^{L_j}_{L_j+N}(x) - S_{L_j}(x)\big| 
		\leq c 
		|S_{L_j}(x)| \cdot \| X_{L_j}(x) \| \left| a_{L_j+N-1} - a_{L_j-1} \right|.
	\]
	Since
	\[
		X_{L_j}  = \sigma \Id + \frac{1}{a_{L_j+N-1}} R_{L_j},
	\]
	by \eqref{eq:50a} and \eqref{eq:52}, we easily obtain
	\[
		\sup_{j \in \NN} \sup_{x \in K} \| X_{L_j}(x) \| < \infty.
	\]
	Therefore, by \eqref{eq:51} and \eqref{eq:11},
	\[
		\sup_{x \in K} \big|g^{L_j}(x) - g(x)\big| \leq c \left| a_{L_j+N-1} - a_{L_j-1} \right|,
	\]
	thus
	\begin{equation}
		\label{eq:57}
		\lim_{j \to \infty}
		\sup_{x \in K} \big| g^{L_j}(x) - g(x) \big| = 0.
	\end{equation}
	Finally, by \eqref{eq:55}, \eqref{eq:39}, and \eqref{eq:57} we obtain
	\begin{equation}
		\label{eq:41}
		\lim_{j \to \infty}
		\sup_{x \in K} \bigg| \mu'_{L_j}(x) - \frac{\sqrt{-h(x)}}{2\pi g(x)} \bigg| = 0,
	\end{equation}
	and the theorem is a consequence of \cite[Proposition 4]{SwiderskiTrojan2019}.
\end{proof}

\begin{corollary}
	\label{cor:1}
	Let the hypothesis of Theorem \ref{thm:3} be satisfied. If
	\begin{equation}
		\label{cor:1:eq:1}
		\lim_{j \to \infty} (a_{(j+1)N+i-1} - a_{jN+i-1}) = 0,
	\end{equation}
	then the sequence $(\mu_{jN+i} : j \in \NN)$ is weakly convergent to the probability measure $\mu$ which is
	absolutely continuous on $K$. Moreover, the sequence $(p_n : n \in \NN_0)$ is orthonormal in $L^2(\RR, \mu)$, and
	\begin{equation}
		\label{eq:38}
		\lim_{j \to \infty} \sup_{x \in K} \big| \mu'_{jN+i}(x) - \nu'(x) \big| = 0.
	\end{equation}
\end{corollary}
\begin{proof}
	In view of \eqref{cor:1:eq:1}, there is $c > 0$ such that for all $k \geq 0$,
	\begin{align*}
		a_{kN+i} 
		&= 
		\sum_{j=0}^{k-1} \big( a_{(j+1)N+i} - a_{jN+i} \big) + a_i \\
		&\leq
		c(k+1).
	\end{align*}
	Therefore,
	\begin{align*}
		\sum_{n = 0}^{k_0N+i} \frac{1}{a_n} 
		&\geq 
		\sum_{k = 0}^{k_0} \frac{1}{a_{kN+i}} \\
		&\geq 
		\frac{1}{c} \sum_{k=1}^{k_0} \frac{1}{k}.
	\end{align*}
	Thus, the Carleman condition is satisfied, and consequently, there is the only one measure $\mu$
	such that $(p_n : n \in \NN_0)$ are orthonormal in $L^2(\RR, \mu)$. Lastly, \eqref{eq:38} is a consequence
	of \eqref{eq:41}.
\end{proof}

\section{Asymptotics of orthogonal polynomials}
\label{sec:4}
Our next goal is to derive the asymptotic formula for the polynomials $(p_n : n \in \NN_0)$.
\begin{theorem}
	\label{thm:1}
	Let $N$ be a positive integer, $\sigma \in \{-1, 1\}$, and $i \in \{0, 1, \ldots, N-1 \}$. We set
	\[
		R_n = a_{n+N-1} ( X_n - \sigma \Id ).
	\]
	Suppose that
	\[
		\lim_{j \to \infty} a_{jN+i-1} = \infty.
	\]
	Let $K$ be a compact interval with non-empty interior contained in
	\[
		\Lambda = 
		\big\{ 
		x \in \RR : \lim_{j \to \infty} \discr R_{jN+i}(x) \ \text{exists and is negative} 
		\big\}.
	\]
	If
	\[
		(X_{jN+i} : j \in \NN), (R_{jN+i} : j \in \NN) \in \calD_1 \big( K, \GL(2, \RR) \big),
	\]
	then there are $M > 0$ and a continuous function $\vphi$ such that
	\begin{equation}
		\label{eq:16} 
		\lim_{k \to \infty} 
		\sup_{x \in K}
		\bigg|
		\frac{a_{(k+1)N+i-1}}{\prod_{j=M+1}^k \lambda_{jN+i}(x)} 
		\Big(p_{(k+1) N+i}(x) -\overline{\lambda_{kN+i}(x)} p_{kN+i}(x) \Big)
		-\vphi(x)
		\bigg| = 0
	\end{equation}
	where
	\[
		\lambda_n(x) = \frac{\tr X_{n}(x)}{2} + \frac{i}{2} \sqrt{-\discr X_n(x)}.
	\]
	Moreover,
	\begin{equation}
		\label{eq:44}
		\sqrt{\frac{a_{(k+1)N+i-1}}{a_{(M+1)N+i-1}}} p_{kN+i}(x) 
		= 
		\frac{2 |\vphi(x)|}{\sqrt{-\discr \calR_i(x)}} 
		\sin \bigg( \sum_{j=M+1}^k \arg \lambda_{jN+i}(x) + \arg \vphi(x) \bigg) 
		+ E_{kN+i}(x)
	\end{equation}
	where $\calR_i$ is the limit of $(R_{jN+i} : j \in \NN)$, and 
	\[
		\sup_K |E_{kN+i}| 
		\leq 
		c \sum_{j=k}^\infty 
		\Big( \sup_K \| X_{(j+1)N+i} - X_{jN+i} \| + \sup_K \| R_{(j+1)N+i} - R_{jN+i} \| \Big).
	\]
\end{theorem}
\begin{proof}
	Let us fix a compact interval $K \subset \Lambda$ with non-empty interior. Since $\calR_i$ is the uniform limit of
	$(R_{jN+i} : j \in \NN)$, there are $\delta > 0$ and $M > 0$ such that for all $x \in K$ and $k \geq M$,
	\[
		\discr R_{kN+i}(x) \leq -\delta, \quad \text{and} \quad
		|[R_{kN+i}(x)]_{1,2}| > \delta
	\]
	Therefore, the matrix $R_{kN+i}$ has two eigenvalues $\xi_{k}$ and $\overline{\xi_{k}}$ where
	\begin{equation}
		\label{eq:26}
		\xi_k(x) = \frac{\tr R_{kN+i}(x)}{2} + \frac{i}{2} \sqrt{-\discr{R_{kN+i}(x)}}.
	\end{equation}
	Let us next observe that for $k \geq M$,
	\[
		\Im \xi_{k}(x) = \tfrac{1}{2} \sqrt{-\discr{R_{kN+i}(x)}} \geq \tfrac{1}{2} \sqrt{\delta}.
	\]
	Moreover,
	\[
		R_{kN+i}(x) = C_k(x) \tilde{D}_k(x) C_k^{-1}(x)
	\]
	where
	\[
		C_k = 
		\begin{pmatrix}
			1 & 1 \\
			\frac{\xi_{kN+i}-[R_{kN+i}]_{1,1}}{[R_{kN+i}]_{1,2}} & \frac{\overline{\xi_{kN+i}}-[R_{kN+i}]_{1,1}}{[R_{kN+i}]_{1,2}}
		\end{pmatrix}
		,
		\qquad\text{and}\qquad
		\tilde{D}_k(x) = 
		\begin{pmatrix}
			\xi_{k} & 0 \\
			0 & \overline{\xi_{k}}
		\end{pmatrix}.
	\]
	Since
	\[
		X_{kN+i} = \sigma \Id + \frac{1}{a_{(k+1)N+i-1}} R_{kN+i},
	\]	
	we obtain
	\[
		X_{kN+i}(x) = C_k(x) D_k(x) C_k^{-1}(x)
	\]
	where
	\[
		D_k = \sigma \Id + \frac{1}{a_{(k+1)N+i-1}} \tilde{D}_k.
	\]
	In particular, $X_{kN+i}$ has two eigenvalues $\lambda_{kN+i}$ and $\overline{\lambda_{kN+i}}$ where
	\[
		\lambda_{kN+i} = \sigma+ \frac{1}{a_{(k+1)N+i-1}} \xi_k.
	\]
	We next set
	\[
		\phi_{k} = \frac{p_{(k+1)N+i} - \overline{\lambda_{kN+i}} p_{kN+i}}{\prod_{j = M+1}^k \lambda_{jN+i}},
	\]
	and claim that the following holds true.
	\begin{claim}
		\label{clm:2}
		There is $c > 0$ such that for all $m \geq n \geq M$, and $x \in K$,
		\[
			\big|a_{(m+1)N+i-1} \phi_{m}(x) - a_{(n+1)N+i-1}\phi_{n}(x) \big|
			\leq
			c \Big(
			\sum_{j = n}^\infty \big| \xi_{j+1}(x) - \xi_{j}(x) \big|
			+
			\sum_{j = n}^\infty \big| \lambda_{(j+1)N+i}(x) - \lambda_{jN+i}(x)\big|\Big).
		\]
	\end{claim}
	We start by writing
	\[
		p_{mN+i} = 
		\left\langle 
		C_{m-1} \Big(\prod_{j = n}^{m-1} D_j C_j^{-1} C_{j-1} \Big) C_{n-1}^{-1}
		\begin{pmatrix}
			p_{nN+i-1} \\
			p_{nN+i}
		\end{pmatrix}
		,
		\begin{pmatrix}
			0 \\
			1
		\end{pmatrix}
		\right\rangle.
	\]
	Let us introduce two auxiliary functions
	\[
		q_m = \left\langle
		C_\infty 
		\Big(\prod_{j = n}^{m-1} D_j \Big) C_{n-1}^{-1}
        \begin{pmatrix}
            p_{nN+i-1} \\
            p_{nN+i}
        \end{pmatrix}
        ,
        \begin{pmatrix}
            0 \\
            1
        \end{pmatrix}
		\right\rangle,
	\]
	and
	\[
		\psi_m = \frac{q_{m+1} - \overline{\lambda_{mN+i}} q_m}{\prod_{j = M+1}^m \lambda_{jN+i}}.
	\]
	Notice that 
	\[
		a_{(m+1)N+i-1} \big( \phi_m - \psi_m \big) = 
		\left\langle
		Y_m 
		\begin{pmatrix}
			p_{nN+i-1} \\
			p_{nN+i}
		\end{pmatrix}
		,
		\begin{pmatrix}
			0 \\
			1
		\end{pmatrix}
		\right\rangle
		\prod_{j = M+1}^m \frac{1}{\lambda_{jN+i}}
	\]
	where
	\begin{align*}
		Y_m 
		&= (C_m - C_{\infty}) D_m^{-1} 
		\Big( a_{(m+1)N+i-1} \big( D_m - \overline{\lambda_{mN+i}} \Id \big) \Big) 
		\Big( \prod_{j=n}^{m} D_j C_j C_{j-1}^{-1} \Big) C_{n-1}^{-1} \\
		&\phantom{=}+ C_\infty D_m^{-1} \Big( a_{(m+1)N+i-1} \big( D_m - \overline{\lambda_{mN+i}} \Id \big) \Big) 
		\Big(\prod_{j=n}^{m} D_j C_j C_{j-1}^{-1} - \prod_{j=n}^{m} D_j \Big) C_{n-1}^{-1}.
	\end{align*}	
	In view of \cite[Propositon 1]{SwiderskiTrojan2019}, we have
	\[
		\norm{Y_m} 
		\lesssim 
		\bigg(\prod_{j = n}^{m} \norm{D_j} \bigg) 
		\Big\| a_{(m+1)N+i-1} \big( D_m - \overline{\lambda_{mN+i}} \Id \big) \Big\|
		\bigg(\sum_{j = n-1}^\infty \norm{\Delta C_j} + \norm{C_\infty - C_{m}}\bigg).
	\]
	Since
	\begin{equation} 
		\label{eq:10}
		a_{(m+1)N+i-1} \big( D_m - \overline{\lambda_{mN+i}} \Id) =
		\begin{pmatrix}
			2 i \Im \xi_{m} & 0 \\
			0 & 0
		\end{pmatrix},
	\end{equation}
	the right-hand side is convergent, hence bounded. Therefore, we have
	\begin{align*}
		\norm{Y_m}
		&\lesssim \bigg(\prod_{j = n}^{m} \norm{D_j} \bigg) \sum_{j = n-1}^\infty 
		\norm{\Delta C_j} \\
		&\lesssim
		\prod_{j=n}^{m} \abs{\lambda_{jN+i}} \cdot 
		\sum_{j=n-1}^\infty \big|\xi_{j+1} - \xi_{j}\big|.
	\end{align*}
	Following the arguments used in the proof of \cite[Claim 2]{SwiderskiTrojan2019}, we conclude that there is $c > 0$
	so that for all $n > M$ and $x \in K$,
	\begin{equation} 
		\label{eq:27}
		\frac{\sqrt{p^2_{nN+i}(x) + p^2_{nN+i-1}(x)}}{\prod_{j = M+1}^{n-1} \abs{\lambda_{jN+i}(x)}}
		\leq c,
	\end{equation}
	and consequently, for all $m \geq n > M$,
	\begin{equation}
		\label{eq:8}
		a_{(m+1)N+i-1}
		\big| \phi_{m} - \psi_m \big|
		\lesssim
		\sum_{j = n-1}^\infty
		\big| \xi_{j+1} - \xi_{j} \big|.
	\end{equation}
	We next notice that
	\begin{align*}
		&
		a_{(m+1)N+i-1} \big( q_{m+1} - \overline{\lambda_{mN+i}} q_m \big) \\
		&\qquad\qquad=
		\left\langle
		C_\infty 
		\Big( a_{(m+1)N+i-1} \big(D_m - \overline{\lambda_{mN+i}} \Id \big) \Big) 
		\Big( \prod_{j = n}^{m-1} D_j \Big) C_{n-1}^{-1}
		\begin{pmatrix}
            p_{nN+i-1} \\
            p_{nN+i}
        \end{pmatrix},
        \begin{pmatrix}
            0 \\
            1
        \end{pmatrix}
		\right\rangle.
	\end{align*}
	Since by \eqref{eq:10}
	\[
		\frac{a_{(m+1)N+i-1}}{\prod_{j = n}^m \lambda_{jN+i}}
		\Big(D_m - \overline{\lambda_{mN+i}} \Id\Big) \prod_{j = n}^{m-1} D_j
		=
		\frac{2 i}{\lambda_{mN+i}}
		\begin{pmatrix}
			\Im \xi_{m} & 0 \\
			0 & 0 
		\end{pmatrix},
	\]
	we obtain
	\begin{align*}
		a_{(m+1)N+i-1}
		\big|\psi_m - \psi_n \big| 
		&\lesssim
		\bigg| \frac{\Im \xi_m}{\lambda_{mN+i}} - \frac{\Im \xi_{n}}{\lambda_{nN+i}} \bigg| \\
		&\leq 
		\sum_{j = n}^\infty 
		\bigg| \frac{\Im \xi_{j+1}}{\lambda_{(j+1)N+i}} - \frac{\Im \xi_{j}}{\lambda_{jN+i}} \bigg|.
	\end{align*}
	Observe that
	\begin{align*}
		\bigg| \frac{\Im \xi_{j+1}}{\lambda_{(j+1)N+i}} - \frac{\Im \xi_{j}}{\lambda_{jN+i}} \bigg|
		&= 
		\bigg| \frac{\lambda_{jN+i} \Im (\xi_{j+1} - \xi_j) - (\lambda_{(j+1)N+i} - \lambda_{jN+i}) \Im \xi_j}
		{\lambda_{(j+1)N+i} \lambda_{jN+i}}\bigg| \\
		&\lesssim
		|\xi_{j+1} - \xi_j| + |\lambda_{(j+1)N+i} - \lambda_{jN+i}|.
	\end{align*}
	Hence,
	\[
		a_{(m+1)N+i-1} \big|\psi_m - \psi_n \big| 
		\lesssim 
		\sum_{j=n}^\infty \Big( \big| \xi_{j+1} - \xi_j \big| + \big| \lambda_{(j+1)N+i} - \lambda_{jN+i} \big| \Big),
	\]
	which together with \eqref{eq:8} implies that for all $m \geq n > M$ and $x \in K$,
	\[
		\big|a_{(m+1)N+i-1} \phi_{m}(x) - a_{(n+1)N+i-1} \phi_{n}(x) \big| 
		\lesssim
		\sum_{j=n}^\infty \Big( \big| \xi_{j+1}(x) - \xi_j(x) \big| 
		+ \big| \lambda_{(j+1)N+i}(x) - \lambda_{jN+i}(x) \big| \Big),
	\]
	proving Claim \ref{clm:2}.

	In particular, Claim \ref{clm:2} entails that the sequence $(a_{(m+1)N+i-1} \phi_{m} : m \in \NN)$ converges. 
	Let us denote by $\vphi$ its limit. Hence,
	\[
		\big| \vphi(x) - a_{(n+1)N+i-1} \phi_{n}(x) \big| 
		\lesssim
		\sum_{j=n}^\infty \Big( \big| \xi_{j+1}(x) - \xi_j(x) \big| 
		+ \big| \lambda_{(j+1)N+i}(x) - \lambda_{jN+i}(x) \big| \Big)
	\]
	Since polynomials $p_n$ are having real coefficients, by taking imaginary part we obtain
	\begin{align*}
		&
		\bigg|
		\frac{a_{(n+1)N+i-1}}{2} \sqrt{-\discr{X_{nN+i}(x)}} \frac{p_{nN+i}(x)}{\prod_{j=M+1}^n |\lambda_{jN+i}(x)|} 
		-
		\abs{\vphi(x)} \sin\Big(\sum_{j = M+1}^n \arg \lambda_{jN+i}(x) + \arg \vphi(x) \Big)
		\bigg| \\
		&\qquad\qquad
		\lesssim
		\sum_{j=n}^\infty \Big( \big| \xi_{j+1}(x) - \xi_j(x) \big| 
		+ \big| \lambda_{(j+1)N+i}(x) - \lambda_{jN+i}(x) \big| \Big).
	\end{align*}
	Observe that
	\[
		\det X_{jN+i} = \prod_{k = jN+i}^{(j+1)N+i-1} \det B_k = \frac{a_{jN+i-1}} {a_{(j+1)N+i-1}},
	\]
	thus
	\begin{align*}
		\prod_{j = M+1}^n \big|\lambda_{jN+i} \big|^2 
		&= \prod_{j = M+1}^n \det X_{jN+i} \\
		&= \frac{a_{(M+1)N+i-1}}{a_{(n+1)N+i-1}}.
	\end{align*}
	Moreover,
	\begin{align*}
		\discr X_{nN+i} &= 
		(\lambda_{nN+i} - \overline{\lambda_{nN+i}})^2 \\
		&= 
		\frac{1}{a_{(n+1)N+i-1}^2} (\xi_{n} - \overline{\xi_n})^2
		=
		\frac{1}{a_{(n+1)N+i-1}^2} \discr R_{nN+i}.
	\end{align*}
	Therefore,
	\begin{align*}
		&\bigg|
		\sqrt{\frac{a_{(n+1)N+i-1}}{a_{(M+1)N+i-1}}} \sqrt{-\discr R_{nN+i}(x)} p_{nN+i}(x) 
		-
		2 |\vphi(x)| 
		\sin\Big(\sum_{j = M+1}^n \arg \lambda_{j N + i}(x) + \arg \vphi(x)\Big)
		\bigg| \\
		&\qquad\qquad\qquad\qquad
		\lesssim
		\sum_{j=n}^\infty \Big( \big| \xi_{j+1} - \xi_j \big| + \big| \lambda_{(j+1)N+i} - \lambda_{jN+i} \big| \Big).
	\end{align*}
	By \eqref{eq:26}, we can write
	\[
		\bigg|
		\frac{1}{\sqrt{-\discr R_{nN+i}(x)}} - \frac{1}{\sqrt{-\discr \calR_i(x)}} \bigg|
		\lesssim
		\sum_{j = n}^\infty \big| \xi_{j+1} - \xi_{j} \big|,
	\]
	thus
	\begin{align*}
		&\bigg| \sqrt{\frac{a_{(n+1)N+i-1}}{a_{(M+1)N+i-1}}} p_{nN+i}(x) - \frac{2 |\vphi(x)|}{\sqrt{-\discr \calR_i(x)}} 
		\sin \bigg( \sum_{j=M+1}^n \arg \lambda_{j N + i}(x) + \arg \vphi(x) \bigg) \bigg| \\
		&\qquad\qquad\qquad\qquad
		\lesssim
		\sum_{j=n}^\infty \Big( \big| \xi_{j+1} - \xi_j \big| + \big| \lambda_{(j+1)N+i} - \lambda_{jN+i} \big| \Big).
	\end{align*}
	Finally, we have
	\[
		\big| \lambda_{(j+1)N+i} - \lambda_{jN+i} \big|
		\lesssim
		\big\|X_{(j+1)N+i} - X_{jN+i} \big\|
	\]
	and
	\[
		\big| \xi_{j+1} - \xi_{j} \big|
		\lesssim
		\big\|R_{(j+1)N+i} - R_{jN+i} \big\|
	\]
	and the theorem follows.
\end{proof}

Our next task is to compute $\abs{\vphi(x)}$. To do this, once again, we use the truncated sequences
defined in \eqref{eq:42a} and \eqref{eq:42b}.  
\begin{theorem} 
	\label{thm:5}
	Let $N$ be a positive integer, $\sigma \in \{-1, 1\}$, and $i \in \{0, 1, \ldots, N-1 \}$. Suppose that
	\[
		\label{eq:22}
		\lim_{n \to \infty} a_{nN+i-1} = \infty \qquad \text{and} \qquad 
		\lim_{n \to \infty} (a_{(n+1)N+i-1} - a_{nN+i-1}) = 0.
	\]
	Let $K$ be a compact interval with non-empty interior contained in
	\[
		\Lambda = \left\{ x \in \RR : \lim_{n \to \infty} \discr R_{nN+i}(x) \ \text{exists and is negative} \right\}
	\]
	where
	\begin{equation} 
		\label{eq:23}
		R_n = a_{n+N-1} ( X_n - \sigma \Id ).
	\end{equation}
	If
	\[
		(X_{nN+i} : n \in \NN), (R_{nN+i} : n \in \NN) \in \calD_1 \big( K; \GL(2, \RR) \big),
	\]
	then the polynomials $(p_n : n \in \NN_0)$ are orthonormal with respect to the measure $\mu$, which
	is purely absolutely continuous on $K$. Moreover, there is $M>0$ and a continuous real-valued function $\eta$ such
	that
	\[
		\sqrt{a_{(n+1)N+i-1}} p_{nN+i}(x) 
		= 
		\sqrt{\frac{2 \big| [\calR_i(x)]_{2, 1} \big|}{\pi \mu'(x) \sqrt{-\discr \calR_i(x)}} }
		\sin \bigg( \sum_{j=M+1}^n \theta_{j}(x) + \eta(x) \bigg)
		+ E_{nN+i}(x)
	\]
	where $\calR_i$ is the limit of $(R_{nN+i} : n \in \NN)$, 
	\[
		\theta_{j}(x) = \arccos \bigg( \frac{\tr X_{jN+i}(x)}{2 \sqrt{\det X_{jN+i}(x)}} \bigg)
	\]
	and 
	\[
		\sup_K |E_{nN+i}| \leq c 
		\sum_{j=n}^\infty 
		\Big( 
		\sup_K \big\| X_{(j+1)N+i} - X_{jN+i} \big\| + 
		\sup_K \big\| R_{(j+1)N+i} - R_{jN+i} \big\| 
		\Big).
	\]
\end{theorem}
\begin{proof}
	Since $K \subset \Lambda$, and $\discr R_{kN+i}$ is a polynomial of degree at most $2N$, there
	are $\delta > 0$ and $M \geq 1$ such that for all $x \in K$ and $n \geq M$,
	\[
		\discr R_{nN+i}(x) \leq -\delta.
	\]
	Let us consider $L \in \{L_k : k \in \NN\}$ where
	\[
		L_k = kN+i.
	\]
	We set
	\[
		\Lambda_L = \big\{x \in \RR : \discr \big(R^L_{L+N}(x) \big) < 0 \big\}
	\]
	where
	\begin{equation} 
		\label{eq:24}
		R^L_n = a^L_{n+N-1} \big( X^L_{n} - \sigma \Id \big),
	\end{equation}
	and $X^L_{n}$ is defined by the formula \eqref{eq:43}. In view of \eqref{eq:42a} and \eqref{eq:42b}, we have
	\begin{equation} 
		\label{eq:21}
		X^{L_k}_{jN+i} = 
		\begin{cases}
			X_{jN+i} & \text{if } 0 \leq j \leq k, \\
			X_{L_k+N}^{L_k} & \text{if } k < j. 
		\end{cases}
	\end{equation}
	Moreover, by Proposition~\ref{prop:1}, there is $L_0 \geq M$ such that $K \subset \Lambda_L$ for all 
	$L \geq L_0$. For $x \in K$ we set
	\begin{align} 
		\label{eq:26a}
		\xi_{m}(x) &= \frac{1}{2} \tr R_{mN+i}(x) + \frac{i}{2} \sqrt{-\discr R_{mN+i}(x)}, \\
		\label{eq:26b}
		\xi_{m}^L(x) &= \frac{1}{2} \tr R^L_{mN+i}(x) + \frac{i}{2} \sqrt{-\discr R^L_{mN+i}(x)},
	\end{align}
	and
	\begin{align}
		\label{eq:25a}
		\lambda_{mN+i}(x) &= \gamma + \frac{1}{a_{(m+1)N+i}} \xi_m(x), \\
		\label{eq:25b}
		\lambda^L_{mN+i}(x) &= \gamma + \frac{1}{a_{(m+1)N+i}} \xi^L_m(x).
	\end{align}
	Lastly, we define
	\begin{align*}
		\phi_{mN+i}(x) &= \frac{p_{(m+1)N+i}(x) - \overline{\lambda_{mN+i}(x)} p_{mN+i}(x)}
		{\prod_{j=M+1}^m \lambda_{jN+i}(x)} \\
		\phi_{mN+i}^L(x) &= \frac{p^L_{(m+1)N+i}(x) - \overline{\lambda_{mN+i}^L(x)} p^L_{mN+i}(x)}
		{\prod_{j = M+1}^m \lambda^L_{jN+i}}
	\end{align*}
	where $(p^L_n : n \in \NN_0)$ is the sequence of orthogonal polynomials corresponding to \eqref{eq:42a} and
	\eqref{eq:42b}.
	\begin{claim}
		\label{clm:3}
		\[
			\lim_{k \to \infty} 
			\sup_{x \in K} 
			\Big| a_{L_k+N-1} \phi^{L_k}_{L_k+N}(x) - \varphi(x) \Big| = 0.
		\]
	\end{claim}
	First, let us observe that, by \eqref{eq:21}, we have
	\begin{align*}
		\phi_{L_k+N} &= 
		\frac{1}{\lambda_{L_k+N} \prod_{j=M+1}^k \lambda_{jN+i}}
		\bigg\langle
		\bigg( X_{L_k+N} - \overline{\lambda_{L_k+N}} \Id \bigg)
		\begin{pmatrix}
			p_{L_k+N-1} \\
			p_{L_k+N}
			\end{pmatrix},
		\begin{pmatrix}
			0 \\ 
			1
		\end{pmatrix}
		\bigg\rangle \\
		\intertext{and}
		\phi^{L_k}_{L_k+N} &= 
		\frac{1}{\lambda^{L_k}_{L_k+N} \prod_{j=M+1}^k \lambda_{jN+i}}
		\bigg\langle
		\bigg( X^{L_k}_{L_k+N} - \overline{\lambda^{L_k}_{L_k+N}} \Id \bigg)
		\begin{pmatrix}
			p_{L_k+N-1} \\
			p_{L_k+N}
		\end{pmatrix},
		\begin{pmatrix}
			0 \\ 
			1
		\end{pmatrix}
		\bigg\rangle.
	\end{align*}
	Thus,
	\[
		a_{L_k+2N-1} \phi_{L_k+N} - a^{L_k}_{L_k+2N-1} \phi^{L_k}_{L_k+N} =
		\frac{1}{\prod_{j=M+1}^k \lambda_{jN+i}}
		\bigg\langle Y_k
		\begin{pmatrix}
			p_{L_k+N-1} \\
			p_{L_k+N}
		\end{pmatrix},
		\begin{pmatrix}
			0 \\ 
			1
		\end{pmatrix}
		\bigg\rangle,
	\]
	where
	\[
		Y_k = 
		\frac{a_{L_k+2N-1}}{\lambda_{L_k+N}} 
		\Big( X_{L_k+N} - \overline{\lambda_{L_k+N}} \Id \Big) -
		\frac{a_{L_K+2N-1}^{L_k}}{\lambda^{L_k}_{L_k+N}} 
		\big( X^{L_k}_{L_k+N} - \overline{\lambda^{L_k}_{L_k+N}} \Id \big).
	\]
	Hence, by \eqref{eq:23}, \eqref{eq:24}, \eqref{eq:25a} and \eqref{eq:25b}
	\[
		Y_k = 
		\frac{1}{\lambda_{L_k+N}} \Big( R_{L_k+N} - \overline{\xi_{L_k+N}} \Id \Big) - 
		\frac{1}{\lambda^{L_k}_{L_k+N}} \Big( R^{L_k}_{L_k+N} - \overline{\xi^{L_k}_{L_k+N}} \Id \Big),
	\]
	and consequently,
	\[
		\| Y_k \| \leq 
		\Bigg|\frac{1}{\lambda_{L_k+N}} - \frac{1}{\lambda^{L_k}_{L_k+N}} \Bigg| \cdot \Big\| R_{L_k+N} \Big\| +
		\frac{1}{\big|\lambda^{L_k}_{L_k+N} \big|} \cdot \Big\| R_{L_k+N} - R^{L_k}_{L_k+N} \Big\| +
		\Bigg| 
		\frac{\overline{\xi^{L_k}_{L_k+N}}}{\lambda^{L_k}_{L_k+N}} 
		- 
		\frac{\overline{\xi_{L_k+N}}}{\lambda_{L_k+N}}
		\Bigg|.
	\]
	Since $(R_{L_k} : k \in \NN)$ is convergent, Proposition~\ref{prop:1} together with \eqref{eq:26a}--\eqref{eq:25b}
	implies that
	\[
		\lim_{k \to \infty} \sup_K \| Y_k \| = 0.
	\]
	Hence, by \eqref{eq:27}
	\[
		\lim_{k \to \infty} \sup_K \Big|a_{L_k+2N-1} \phi_{L_k+N} - a^{L_k}_{L_k+2N-1} \phi^{L_k}_{L_k+N}\Big| = 0,
	\]
	and the conclusion follows by \eqref{eq:16}.
	\begin{claim} 
		\label{clm:4}
		For $x \in K$, we have
		\[
			\big| \vphi(x) \big|^2 =
			\frac{| [\calR_i(x)]_{2, 1} | \sqrt{-\discr \calR_i(x)}} {2 \pi \mu'(x) a_{(M+1)N+i-1}}.
		\]
	\end{claim}
	In view of \cite[Claim 5]{SwiderskiTrojan2019} we can write
	\[
		\big| \phi_{L+N}^L(x) \big|^2 =
		\frac{| [ X_{L+N}^L(x) ]_{2, 1} | \sqrt{-\discr\big( X_{L+N}^L(x) \big)}}
		{2 \pi \mu_L'(x) a_{(M+1)N+i-1}}.
	\]
	Hence,
	\[
		\big| a^L_{L+2N-1} \phi_{L+N}^L(x) \big|^2 = 
		\frac{|[R_{L+N}^L(x)]_{2, 1} | \sqrt{-\discr\big( R_{L+N}^L(x) \big)}}
		{2 \pi \mu_L'(x) a_{(M+1)N+i-1}}
	\]
	which, by Proposition~\ref{prop:1} and Corollary~\ref{cor:1}, approaches to 
	\[
		\frac{| [ \calR_i(x) ]_{2, 1} | \sqrt{-\discr \calR_i(x) }}
		{2 \pi \mu'(x) a_{(M+1)N+i-1}}
	\]
	as $L$ tends to infinity, uniformly with respect to $x \in K$.
	Thus, the conclusion follows by Claim~\ref{clm:3}.

	Now, to finish the proof of the theorem it is enough to combine Claim \ref{clm:4} with \eqref{eq:44}.
\end{proof}

\section{Christoffel functions in residue classes}
\label{sec:5}
\begin{lemma} 
	\label{lem:1}
	Let $(\gamma_n : n \in \NN)$ be a sequence of positive numbers such that
	\[
		\sum_{n=0}^\infty \gamma_n = \infty, \qquad\text{and}\qquad
		\lim_{n \to \infty} \gamma_n = 0.
	\]
	Suppose that there are a compact set $K \subset \RR^d$, and $(\xi_n : n \in \NN)$ a sequence of real functions
	on $K$ such that
	\[
		\lim_{n \to \infty} \sup_{x \in K} \bigg| \frac{\xi_n(x)}{\gamma_n} - \psi(x) \bigg| = 0
	\] 
	for some function $\psi: K \rightarrow [0, \infty)$ satisfying
	\[
		c^{-1} \leq \psi(x) \leq c, \qquad\text{for all } x \in K.
	\]
	We set
	\[
		\Xi_n(x) = \sum_{j = 0}^n \xi_j(x), \qquad\text{and}\qquad
		\Gamma_n = \sum_{j = 0}^n \gamma_j.
	\]
	Then for any $f \in \calC^1([0, \infty))$ such that both $f$ and $f'$ are bounded on $[0, \infty)$,
	\[
		\lim_{n \to \infty}
		\frac{1}{\Gamma_n}
		\sum_{k=0}^n \gamma_k
		f (\Xi_k(x))
		=
		\lim_{n \to \infty}
		\frac{1}{\Xi_n(x)} \int_0^{\Xi_n(x)} f(t) \ud t
	\]
	uniformly with respect to $x \in K$ provided that the right-hand side exists.
\end{lemma}
\begin{proof}
	Since $f$ and $1/\psi$ are bounded on $[0, \infty)$ and $K$ respectively, by the Stolz--Ces\'aro theorem we get 
	\begin{align*}
		\lim_{n \to \infty} 
		\frac{1}{\Gamma_n} 
		\sum_{k = 0}^n \bigg|\gamma_k - \frac{\xi_k(x)}{\psi(x)} \bigg| \cdot \big|f\big(\Xi_k(x)\big)\big|
		=
		\lim_{n \to \infty}
		\bigg|\psi(x) - \frac{\xi_n(x)}{\gamma_n}\bigg| \cdot \bigg|\frac{f\big(\Xi_n(x)\big)}{\psi(x)}\bigg|
		=0
	\end{align*}
	uniformly with respect to $x \in K$. Therefore,
	\begin{equation}
		\label{eq:4}
		\lim_{n \to \infty} \frac{1}{\Gamma_n} \sum_{k = 0}^n \gamma_k f\big(\Xi_k(x)\big)
		=
		\lim_{n \to \infty} \frac{1}{\Gamma_n} \sum_{k = 0}^n \frac{\xi_k(x)}{\psi(x)} f\big(\Xi_k(x)\big).
	\end{equation}
	We next observe that, by the mean value theorem, we obtain
	\begin{align*}
		\bigg| \sum_{k=0}^n \frac{\xi_k}{\psi} f(\Xi_k)  -
		\frac{1}{\psi} \int_{0}^{\Xi_n} f(t) \ud t \bigg|
		&\leq
		\sum_{k=0}^n \frac{1}{\psi} \int_{\Xi_{k-1}}^{\Xi_k} 
		\big| f ( \Xi_k) - f (t) \big| \ud t \\
		&\leq
		\sup_{t \in [0, \infty)}{| f'(t)|} \cdot \frac{1}{2} \sum_{k=0}^n \frac{\xi_k^2}{\psi}.
	\end{align*}
	Since by the Stolz--Ces\'aro theorem,
	\begin{align*}
		\lim_{n \to \infty} \frac{1}{\Gamma_n} \sum_{k=0}^n \frac{\xi_k^2(x)}{\psi(x)} 
		&=
		\lim_{n \to \infty} \frac{\xi_n^2(x)}{\gamma_n \psi(x)} \\
		&=
		\lim_{n \to \infty} \gamma_n \psi(x) \bigg(\frac{\xi_n(x)}{\gamma_n \psi(x)}\bigg)^2 =  0
	\end{align*}
	uniformly with respect to $x \in K$, we conclude that
	\[
		\lim_{n \to \infty} \frac{1}{\Gamma_n} \sum_{k=0}^n \frac{\xi_k}{\psi}
		f \big( \Xi_k \big) =
		\lim_{n \to \infty} \frac{1}{\psi \Gamma_n} \int_{0}^{\Xi_n} f(t) \ud t,
	\]
	which together with \eqref{eq:4} implies that
	\begin{equation}
		\label{eq:47}
		\lim_{n \to \infty} \frac{1}{\Gamma_n} \sum_{k=0}^n \gamma_k f \big( \Xi_k \big)
		=
		\lim_{n \to \infty} \frac{1}{\psi \Gamma_n} \int_{0}^{\Xi_n} f(t) \ud t
	\end{equation}
	uniformly on $K$. In view of the Stolz--Ces\'aro theorem, we obtain
	\begin{equation}
		\label{eq:48}
		\lim_{n \to \infty} \frac{\Xi_n(x)}{\psi(x) \Gamma_n} = 
		\lim_{n \to \infty} \frac{\xi_n(x)}{\psi(x) \gamma_n} = 1,
	\end{equation}
	which combined with \eqref{eq:47} completes the proof.
\end{proof}

\begin{lemma}
	\label{lem:3}
	Assume that $(\gamma_n : n \in \NN_0)$ is a sequence of positive numbers such that
	\[
		\sum_{n=0}^\infty \gamma_n = \infty, \qquad\text{and}\qquad
		\lim_{n \to \infty} \frac{\gamma_{n-1}}{\gamma_n} = 1.
	\]
	Let $K$ be compact subset of $\RR$. Suppose that $(\xi_n : n \in \NN_0)$ is a sequence of 
	real functions on $K$ such that
	\[
		\lim_{n \to \infty} \sup_{x \in K} \Big| \frac{\xi_n(x)}{\gamma_n} - \psi(x) \Big| = 0
	\]
	for some function $\psi: K \rightarrow [0, \infty)$ satisfying
	\[
		c^{-1} \leq \psi(x) \leq c, \qquad \text{for all } x \in K.
	\]
	We set
	\[
		\Xi_n(x) = \sum_{k = 0}^n \xi_k(x),\qquad\text{and}\qquad
		\Gamma_n = \sum_{k = 0}^n \gamma_k.
	\]
	If $\{f_s : s \in S\}$ is a bounded subset of $\calC^1 \big( [-\epsilon, \epsilon + \sup_{K} \psi] \big)$ for a certain $\epsilon > 0$ and some index set $S$, then
	\begin{equation}
		\label{eq:60}
		\lim_{n \to \infty} 
		\frac{1}{\Gamma_n}
		\sum_{k=0}^n 
		\gamma_k
		f_s \big( \Gamma_n^{-1} \Xi_k(x)\big) = 
		\frac{1}{\psi(x)} \int_0^{\psi(x)} f_s(t) \ud t
	\end{equation}
	uniformly with respect to $x \in K$ and $s \in S$.
\end{lemma}
\begin{proof}
	First, let us observe that by the Stolz--Ces\'aro theorem
	\begin{equation}
		\label{eq:61}
		\lim_{n \to \infty} \frac{\Xi_n(x)}{\Gamma_n} = 
		\lim_{n \to \infty} \frac{\xi_n(x)}{\gamma_n} = \psi(x)
	\end{equation}
	uniformly with respect to $x \in K$.

	Let $U = (-\epsilon, \epsilon + \sup_{K} \psi)$. Notice that $\{ f_s : s \in S \}$ is a bounded
	subset of $\calC^1(\overline{U})$. In view of \eqref{eq:61}, there is $M$ such that for every $n \geq M$, $x \in K$,
	and all $k \in \{M, 1, \ldots, n\}$,
	\[
		\frac{\Xi_k(x)}{\Gamma_n} \in U.
	\]
	Let $n \geq M$. Notice that
	\[
		\sup_{x \in K} 
		\bigg| 
		\sum_{k=0}^n \gamma_k f_s \big(\Gamma_n^{-1} \Xi_k(x) \big) 
		-
		\frac{1}{\psi(x)} \sum_{k=0}^n \xi_k(x) f_s \big( \Gamma_n^{-1} \Xi_k(x) \big) 
		\bigg|
		\leq
		c
		\sup_{t \in \overline{U}} |f_s(t)| \cdot
		\sum_{k=0}^n \gamma_k \sup_{x \in K} \bigg| \psi(x) - \frac{\xi_k(x)}{\gamma_k} \bigg|.
	\]
	Since, by the Stolz--Ces\'aro theorem
	\[
		\lim_{n \to \infty} \frac{1}{\Gamma_n}
		\sum_{k=0}^n \gamma_k \sup_{x \in K} {\bigg| \psi(x)  - \frac{\xi_k(x)}{\gamma_k} \bigg|}
		=
		\lim_{n \to \infty} 
		\sup_{x \in K}{\bigg| \psi(x) - \frac{\xi_n(x)}{\gamma_n}\bigg|} = 0,
	\]
	we obtain
	\begin{equation}
		\label{eq:62}
		\lim_{n \to \infty} \frac{1}{\Gamma_n} \sum_{k=0}^n \gamma_k
		f_s \big( \Gamma_n^{-1} \Xi_k(x)\big)=
		\lim_{n \to \infty}
		\frac{1}{\Gamma_n}
		\sum_{k=0}^n \xi_k(x)
		f_s \big( \Gamma_n^{-1} \Xi_k(x)\big)
	\end{equation}
	uniformly with respect to $x \in K$ and $s \in S$. Next, we are going to replace the sum by 
	the integral. By the mean value theorem, we can write
	\begin{align*}
		\bigg| \frac{1}{\Gamma_n} \sum_{k=0}^n \xi_k(x) f_s \big(\Gamma_n^{-1} \Xi_k(x)\big)
		-
		\int_{0}^{\Xi_n(x)/\Gamma_n}  f_s(t) \ud t \bigg|
		&=
		\bigg|
		\sum_{k=0}^n \int_{\Xi_{k-1}(x)/\Gamma_n}^{\Xi_k(x)/\Gamma_n} 
		f_s \big( \Gamma_n^{-1} \Xi_k(x) \big) - f_s (t) \ud t \bigg|\\
		&\leq 
		\frac{1}{2}
		\sup_{t \in \overline{U}}{| f_s'(t)|}
		\frac{1}{2 \Gamma_n^2}
		\sum_{k=0}^n \xi_k^2(x).
	\end{align*}
	By repeated application of the Stolz--Ces\'aro theorem, we get
	\[
		\lim_{n \to \infty} \frac{\sum_{k=0}^n \xi_k^2(x)}{\Gamma_n^2} = 
		\lim_{n \to \infty} \frac{\xi_n^2(x)}{\gamma_n (\Gamma_n + \Gamma_{n-1})} 
		\leq \psi^2(x) \lim_{n \to \infty} \frac{\gamma_n}{\Gamma_n} =
		\psi^2(x) \lim_{n \to \infty} \frac{\gamma_{n} - \gamma_{n-1}}{\gamma_n} = 0
	\]
	uniformly with respect to $x \in K$. Therefore,
	\[
		\lim_{n \to \infty} \frac{1}{\Gamma_n} \sum_{k=0}^n \xi_k(x) 
		f_s \big( \Gamma_n^{-1} \Xi_k(x)\big) = 
		\lim_{n \to \infty} \int_{0}^{\Xi_n(x)/\Gamma_n} f_s(t) \ud t
	\]
	uniformly with respect to $x \in K$ and $s \in S$, which together with \eqref{eq:62} implies that
	\[
		\lim_{n \to \infty} 
		\frac{1}{\Gamma_n} \sum_{k=0}^n \gamma_k f_s \big( \Gamma_n^{-1} \Xi_k(x)\big)
		=
		\frac{1}{\psi(x)} 
		\lim_{n \to \infty} \int_{0}^{\Xi_n(x)/\Gamma_n} f_s(t) \ud t
	\]
	uniformly with respect to $x \in K$ and $s \in S$. Now, by \eqref{eq:61} we conclude \eqref{eq:60}.
\end{proof}

\begin{proposition} 
	\label{prop:3}
	Let $N$ be a positive integer, $\sigma \in \{-1, 1\}$ and $i \in \{0, 1, \ldots, N-1 \}$. Suppose that
	\begin{equation} 
		\label{eq:70}
		\lim_{j \to \infty} a_{jN+i-1} = \infty.
	\end{equation}
	Let $K$ be a compact interval with non-empty interior contained in
	\[
		\Lambda = \left\{ x \in \RR : \lim_{j \to \infty} \discr R_{jN+i}(x) \ \text{exists and is negative} \right\}
	\]
	where
	\[
	        R_n = a_{n+N-1} ( X_n - \sigma \Id ).
	\]
	Assume that for each $x \in K$, the sequence $(R_{jN+i}(x) : j \in \NN_0)$ converges to $\calR_i(x)$.
	Then
	\[
		\lim_{j \to \infty} a_{(j+1)N+i-1} \cdot \arccos\bigg(\sigma \frac{\tr X_{jN+i}(x)}{2 \sqrt{\det X_{jN+i}(x)}}\bigg)
		=\frac{1}{2} \sqrt{-\discr \calR_i(x)}
	\]
	uniformly with respect to $x \in K$.
\end{proposition}
\begin{proof}
	Let $K$ be a compact subset of $\Lambda$. Since each entry in $R_n(x)$ is a polynomial of degree at most $N$,
	\[
		\lim_{j \to \infty} R_{jN+i}(x) = \calR_i(x)
	\]
	uniformly with respect to $x \in K$. Hence, by \eqref{eq:70},
	\[
		\lim_{j \to \infty} X_{jN+i}(x) = \sigma \Id
	\]
	uniformly with respect to $x \in K$. In particular,
	\[
		\lim_{j \to \infty} \frac{\tr X_{jN+i}(x)}{2 \sqrt{\det X_{jN+i}(x)}}
		=
		\sigma.
	\]
	Since
	\[
		\lim_{t \to 1^-} \frac{\arccos t}{\sqrt{1 - t^2}} = 1,
	\]
	we obtain
	\[
		\lim_{j \to \infty}
		\bigg(
		1 - \bigg(\frac{\tr X_{jN+i}(x)}{2 \sqrt{\det X_{jN+i}(x)}} \bigg)^2\bigg)^{-1/2}
		\arccos\bigg(\sigma \frac{\tr X_{jN+i}(x)}{2 \sqrt{\det X_{jN+i}(x)}}\bigg) = 1.
	\]
	Finally, let us observe that
	\[
		\sqrt{1 - \bigg(\frac{\tr X_n(x)}{2 \sqrt{\det X_n(x)}} \bigg)^2} =
		\frac{\sqrt{-\discr X_{n}(x)}}{2 \sqrt{\det X_{n}(x)}} =
		\frac{1}{a_{n+N-1}} \frac{\sqrt{-\discr R_n(x)}}{2 \sqrt{\det X_n(x)}}.
	\]
	Hence,
	\[
		\lim_{j \to \infty} a_{(j+1)N+i-1} \sqrt{1 - \bigg(\frac{\tr X_{jN+i}(x)}{2 \sqrt{\det X_{jN+i}(x)}} \bigg)^2}
		=
		\lim_{j \to \infty} \frac{\sqrt{-\discr R_{jN+i}(x)}}{2 \sqrt{\det X_{jN+i}(x)}}
		=
		\frac{1}{2} \sqrt{-\discr \calR_i(x)},
	\]
	which concludes the proof.
\end{proof}

For $i \in \{0, 1, \ldots, N-1\}$ and $n \in \NN$ we set
\[
	K_{i; n}(x, y) = \sum_{j = 0}^n p_{jN+i}(x) p_{jN+i}(y), 
	\qquad x, y \in \RR,
\]
and
\[
	\rho_{i; n} = \sum_{j = 0}^n \frac{1}{a_{jN+i}}.
\]
We can now state one of the main results of this article.
\begin{theorem} 
	\label{thm:6}
	Let $N$ be a positive integer, $\sigma \in \{-1, 1\}$, and $i \in \{0, 1, \ldots, N-1 \}$. Suppose that
	\begin{equation} 
		\label{thm:6:eq:1}
		\lim_{n \to \infty} a_{nN+i-1} = \infty, \qquad \text{and} \qquad
		\lim_{n \to \infty} (a_{(n+1)N+i-1} - a_{nN+i-1}) = 0.
	\end{equation}
	Let $K$ be a compact interval with non-empty interior contained in
	\[
		\Lambda = 
		\left\{ x \in \RR : \lim_{n \to \infty} \discr R_{nN+i}(x) \ \text{exists and is negative} \right\}
	\]
	where
	\[
		R_n = a_{n+N-1} ( X_n - \sigma \Id ).
	\]
	If
	\[
		(X_{jN+i} : j \in \NN), (R_{jN+i} : j \in \NN) \in \calD_1 \big( K; \GL(2, \RR) \big)
	\]
	then
	\[
		\lim_{n \to \infty} \frac{1}{\rho_{i-1; n}} K_{i; n}(x, x) = 
		\frac{1}{\pi \mu'(x)} 
		\frac{|[\calR_i(x)]_{2, 1}|}{\sqrt{-\discr \calR_i(x)}}
	\]
	uniformly with respect to $x \in K$, where $\calR_i$ is the limit of $(R_{jN+i}: j \in \NN)$.
\end{theorem}
\begin{proof}
	Let $K \subset \Lambda$ be a compact interval with non-empty interior. By Theorem \ref{thm:5}, there is $M > 0$
	such that for all $k \geq M$,
	\[
		a_{(k+1)N+i-1} p_{kN+i}^2(x) 
		= 
		\frac{2 |[\calR_i(x)]_{2,1}|}{\pi \mu'(x) \sqrt{-\discr \calR_i(x)}}
		\sin^2\Big(\eta(x) + \sum_{j = M+1}^k \theta_j(x) \Big) + E_{kN+i}(x)
	\]
	where
	\begin{equation}
		\label{eq:46}
		\lim_{k \to \infty} \sup_{x \in K}{|E_{kN+i}(x)|} = 0. 
	\end{equation}
	Therefore,
	\begin{align*}
		\sum_{k = M}^n p_{kN+i}^2(x) 
		&= \frac{2 |[\calR_i(x)]_{2,1}|}{\pi \mu'(x) \sqrt{-\discr \calR_i(x)}}
		\sum_{k = M}^n \frac{1}{a_{(k+1)N+i-1}}
		\sin^2\Big(\eta(x) + \sum_{j = M+1}^k \theta_{j}(x)\Big)\\
		&\phantom{=}+
		\sum_{k=M}^n \frac{1}{a_{(k+1)N+i-1}} E_{kN+i}(x).
	\end{align*}
	Observe that there is $c > 0$ such that
	\[
		\sup_{x \in K} \sum_{k = 0}^{M-1} p_{kN+i}^2(x) \leq c.
	\]
	By \eqref{thm:6:eq:1},
	\[
		a_{jN+i-1} 
		= a_{i-1} + \sum_{k = 1}^j ( a_{kN+i-1} - a_{(k-1)N+i-1} ) \\
		\leq c(j+1),
	\]
	thus
	\begin{equation}
		\label{eq:36}
		\lim_{n \to \infty} \rho_{i-1;n} = \infty.
	\end{equation}
	Next, by the Stolz--Ces\'aro theorem and \eqref{eq:46}, we obtain
	\[
		\lim_{n \to \infty} \frac{1}{\rho_{i-1; n}} \sum_{k = M}^n \frac{1}{a_{(k+1)N+i-1}} E_{kN+i}(x)
		=
		\lim_{n \to \infty} \frac{a_{nN+i-1}}{a_{(n+1)N+i-1}} E_{nN+i}(x) = 0,
	\]
	since \eqref{thm:6:eq:1} entails that
	\[
		\lim_{n \to \infty} \frac{a_{nN+i-1}}{a_{(n+1)N+i-1}} = 1.
	\]
	Therefore, in view of \eqref{eq:36}
	\begin{equation}
		\label{eq:49}
		\begin{aligned}
		&
		\lim_{n \to \infty} \frac{1}{\rho_{i-1; n}} K_{i; n}(x, x) \\
		&\qquad\qquad=
		\frac{2 |[\calR_i(x)]_{2,1}|}{\pi \mu'(x) \sqrt{-\discr \calR_i(x)}}
		\cdot \lim_{n \to \infty}
		\frac{1}{\rho_{i-1; n}}
		\sum_{k = 0}^n
		\frac{1}{a_{(k+1)N+i-1}} \sin^2\Big(\eta(x) + \sum_{j = M+1}^k \theta_j(x)\Big).
		\end{aligned}
	\end{equation}
	Since $\sin^2(k \pi + x) = \sin^2(x)$, we have
	\begin{equation}
		\label{eq:45}
		\sin^2\Big(\eta(x) + \sum_{j = M+1}^k \theta_j(x)\Big)
		=
		\sin^2\Big(-\eta(x) + \sum_{j = M+1}^k \big(\pi - \theta_j(x)\big)\Big).
	\end{equation}
	Therefore, by taking
	\[
		\gamma_j = \frac{1}{a_{jN+i-1}}, \qquad
		\psi(x) = \frac{1}{2} \sqrt{-\discr \calR_i(x)},
	\]
	and 
	\[
		\xi_0(x) = 
		\sigma \eta(x) + 
		\begin{cases}
			\theta_0(x) & \text{if } \sigma = 1, \\
			\pi - \theta_0(x) & \text{if } \sigma = -1,
		\end{cases}
		\quad\text{and}\quad
		\xi_j(x) = 
		\begin{cases}
			\theta_j(x) & \text{if } \sigma = 1, \\
			\pi - \theta_j(x) & \text{if } \sigma = -1,
		\end{cases}
	\]
	by Proposition \ref{prop:3} we obtain
	\[
		\lim_{j \to \infty} \frac{\xi_j(x)}{\gamma_j} = \psi(x). 
	\]
	Hence, in view of Lemma \ref{lem:1} and \eqref{eq:45} we get
	\begin{align*}
		\lim_{j \to \infty} \frac{1}{\rho_{i-1; n}} \sum_{k = 0}^n \frac{1}{a_{(k+1)N+i-1}} 
		\sin^2\Big(\sum_{j = 0}^k \xi_j(x)\Big)
		&=
		\lim_{n \to \infty} \frac{1}{\Xi_n(x)} \int_0^{\Xi_n(x) } \sin^2(t) \ud t \\
		&=
		\lim_{n \to \infty} \frac{1}{2} - \frac{1}{4 \Xi_n(x)} \sin\big(2 \Xi_n(x)\big).
	\end{align*}
	Lastly, by \eqref{eq:48},
	\[
		\lim_{n \to \infty} \frac{1}{\Xi_n(x)} = 0
	\]
	uniformly with respect to $x \in K$, thus
	\[
		\lim_{j \to \infty} \frac{1}{\rho_{i-1; n}} \sum_{k = 0}^n \frac{1}{a_{(k+1)N+i-1}}
		\sin^2\Big(\sum_{j = 0}^k \xi_j(x)\Big) 
		=
		\frac{1}{2},
	\]
	which together with \eqref{eq:49} finishes the proof.
\end{proof}

\section{Periodic modulations}
\label{sec:6}
\subsection{Definitions and basic properties}
We say that Jacobi parameters $(a_n)$ and $(b_n)$ are \emph{$N$-periodically modulated}
if there are two $N$-periodic sequences
$(\alpha_n : n \in \ZZ)$ and $(\beta_n : n \in \ZZ)$ of positive and real numbers, respectively, 
such that
\begin{enumerate}[(a)]
	\item
	$\begin{aligned}[b]
	\lim_{n \to \infty} a_n = \infty
	\end{aligned},$
	\item
	$\begin{aligned}[b]
	\lim_{n \to \infty} \bigg| \frac{a_{n-1}}{a_n} - \frac{\alpha_{n-1}}{\alpha_n} \bigg| = 0
	\end{aligned},$
	\item
	$\begin{aligned}[b]
	\lim_{n \to \infty} \bigg| \frac{b_n}{a_n} - \frac{\beta_n}{\alpha_n} \bigg| = 0
	\end{aligned}.$
\end{enumerate}

By $(\mathfrak{p}_n : n \in \NN_0)$ we denote orthogonal polynomials associated with sequences $(\alpha_n : n \in \NN_0)$
and $(\beta_n : n \in \NN_0)$. We set 
\[
	\frakX_n(x) = 
	\prod_{j=n}^{n+N-1} \frakB_j(x), \quad \text{where} \quad
	\frakB_j(x) = 
	\begin{pmatrix}
		0 & 1 \\
		-\frac{\alpha_{j-1}}{\alpha_j} & \frac{x - \beta_j}{\alpha_j}
	\end{pmatrix}.
\]
In this article we are interested in the case when 
\[
	\frakX_0(0) = \sigma \Id
\]
for some $\sigma \in \{ -1, 1 \}$ and for any $i \in \{0, 1, \ldots N-1 \}$, the limit
\begin{equation} 
	\label{eq:31a}
	\calR_i(x) = \lim_{n \to \infty} a_{(n+1)N+i-1} \big( X_{nN+i}(x) - \sigma \Id \big)
\end{equation}
exists.

\begin{proposition}
	\label{prop:5}
	If for some $i$ the limit \eqref{eq:31a} exists, then it exists for all $i \in \NN_0$. Moreover, 
	\begin{equation}
		\label{eq:53}
		\calR_{i+1}(x) = \frac{\alpha_i}{\alpha_{i-1}} \frakB_i(0) \calR_i(x) \frakB^{-1}_i(0).
	\end{equation}
\end{proposition}
\begin{proof}
	It is enough to prove \eqref{eq:53}. Observe that
	\begin{align*}
		&a_{(n+1)N+i} \big( X_{nN+i+1}(x) - \sigma \Id \big) \\
		&\qquad\qquad= \frac{a_{(n+1)N+i}}{a_{(n+1)N+i-1}} B_{(n+1)N+i}(x)
		\Big( a_{(n+1)N+i-1} \big( X_{nN+i}(x) - \sigma \Id \big) \Big) 
		B_{nN+i}^{-1}(x).
	\end{align*}
	Computing limits of both sides gives
	\[
		\lim_{n \to \infty} a_{(n+1)N+i} \big( X_{nN+i+1}(x) - \sigma \Id \big) = 
		\frac{\alpha_i}{\alpha_{i-1}} \frakB_i(0) \calR_i(x) \frakB^{-1}_i(0).
	\]
	Hence, we obtain the existence of $\calR_{i+1}$ and the formula \eqref{eq:53} follows.
\end{proof}

We define
\begin{equation}
	\label{eq:32}
	\Lambda = \bigcap_{i=0}^{N-1} \big\{x \in \RR : \discr \calR_i(x) < 0\big\}.
\end{equation}
In view of Proposition~\ref{prop:5}
\begin{equation} \label{eq:32a}
	\Lambda = \big\{ x \in \RR : \discr \calR_0(x) < 0 \big\}.
\end{equation}
We set
\begin{equation} 
	\label{eq:33}
	\upsilon(x) = 
	\frac{1}{N \pi} \sum_{i=0}^{N-1} 
	\frac{1}{\alpha_{i-1}} \frac{\big| [\calR_i(x)]_{2,1} \big|}{\sqrt{-\discr{\calR_i(x)}}},
	\qquad
	x \in \Lambda.
\end{equation}
Assume that there are  $N$-periodic sequences $(s_n : n \in \NN)$ and $(z_n : n \in \NN)$ such that
\begin{equation}
	\label{eq:30}
	\lim_{n \to \infty} \bigg| \frac{\alpha_{n-1}}{\alpha_n} a_n - a_{n-1} - s_n \bigg| = 0, \quad
	\lim_{n \to \infty} \bigg| \frac{\beta_n}{\alpha_n} a_n - b_n - z_n \bigg| = 0.
\end{equation}
Then according to \cite[Proposition 9]{PeriodicIII} the limit \eqref{eq:31a} exists. Moreover, by
\cite[Corollary 1]{PeriodicIII}, there is a compact interval $I$ (possibly empty) such that
\[
	\Lambda = \RR \setminus I.
\]

\begin{proposition}
	\label{prop:2}
	Let $(a_n : n \in \NN_0)$ and $(b_n : n \in \NN_0)$ be $N$-periodically modulated. 
	Suppose that there is $N$-periodic sequence $(s_n : n \in \NN_0)$, such that
	\[
		\lim_{n \to \infty} \bigg| \frac{\alpha_{n-1}}{\alpha_n} a_n - a_{n-1} - s_n \bigg| = 0.
	\]
	Then for each $i \in \{0, 1, \ldots, N-1 \}$,
	\[
		\lim_{k \to \infty} \big( a_{(k+1)N+i} - a_{kN+i} \big) = 
		\alpha_i \sum_{j=0}^{N-1} \frac{s_{j+1}}{\alpha_j}.
	\]
\end{proposition}
\begin{proof}
	Since
	\[
		\sum_{j=0}^{N-1} \bigg( \frac{a_{n+j+1}}{\alpha_{n+j+1}} - \frac{a_{n+j}}{\alpha_{n+j}} \bigg) = 
		\frac{a_{n+N}}{\alpha_{n+N}} - \frac{a_n}{\alpha_n},
	\]
	by $N$-periodicity of $\alpha$ we obtain
	\[
		a_{n+N} - a_n = 
		\alpha_n \sum_{j=0}^{N-1} 
		\frac{1}{\alpha_{n+j}} \bigg( \frac{\alpha_{n+j}}{\alpha_{n+j+1}} a_{n+j+1} - a_{n+j} \bigg).
	\]
	Thus, for $i \in \{0, 1, \ldots, N-1 \}$,
	\[
		\lim_{k \to \infty} \big( a_{(k+1)N+i} - a_{kN+i} \big) =
		\alpha_{i} \sum_{j=0}^{N-1} 
		\frac{s_{j+i+1}}{\alpha_{i+j}} = \alpha_i \sum_{j=0}^{N-1} \frac{s_{j+1}}{\alpha_j},
	\]
	which finishes the proof.
\end{proof}

\subsection{The function $\upsilon$}
\begin{theorem}
	\label{thm:2}
	Let $N$ be a positive integer and $\sigma \in \{-1, 1\}$. Let $(a_n : n \in \NN_0)$ and $(b_n : n \in \NN_0)$ be 
	$N$-periodically modulated Jacobi parameters so that $\frakX_0(0) = \sigma \Id$. Suppose that 
	\begin{equation}
		\label{thm:2:eq:1}
		\lim_{n \to \infty}
		\bigg|
		\frac{\alpha_{n-1}}{\alpha_n} a_n - a_{n-1}
		\bigg| = 0,
		\qquad\text{and}\qquad
		\lim_{n \to \infty}
		\bigg|
		\frac{\beta_n}{\alpha_n} a_n - b_n
		\bigg|
		= 0.
	\end{equation}
	Then $\Lambda = \RR \setminus \{ 0 \}$, and
	\begin{equation}
		\label{thm:2:eq:2}
		\upsilon(x) = \omega'(0), \qquad x \in \Lambda
	\end{equation}
	where $\omega'(x)$ is the version of the density of the equilibrium measure of
	\[
		E = \big\{ x \in \RR : |\tr \frakX_0(x)| \leq 2 \big\}
	\]
	which is continuous on $\text{int}(E)$.
\end{theorem}
\begin{proof}
	First of all, by \cite[Theorem IV.2.5, pp. 216]{Saff1997} (see also the proof of 
	\cite[Corollary 5.4.6]{Simon2010Book}), the density of $\omega$ is continuous on $\text{int}(E)$.

	Let $(\tilde{a}_k : k \in \NN)$ be a positive sequence tending to infinity. By \eqref{thm:2:eq:1}, we can apply
	\cite[Proposition 8 and 9]{PeriodicIII} to conclude that
	\begin{equation}
		\label{eq:6}
		\frac{1}{\alpha_{i-1}} \calR_i(x) = 
		\lim_{k \to \infty} \tilde{a}_k \bigg( \frakX_i \Big( \frac{x}{\tilde{a}_k} \Big) - \sigma \Id \bigg).
	\end{equation}
	We set
	\[
		x_k = \frac{x}{\tilde{a}_k}.
	\]
	By \cite[Proposition 13]{PeriodicIII} we see that $0$ belongs to $\text{int}(E)$, and for all sufficiently large
	$k$
	\[
		\discr \frakX_i (x_k) < 0.
	\]
	Next, by \cite[formula (3.1)]{ChristoffelI}, we have
	\begin{equation}
		\label{eq:7}
		\pi \omega'(x_k) = 
		\frac{1}{N} 
		\sum_{i=0}^{N-1} \frac{1}{\alpha_{i-1}} \frac{|[\frakX_i (x_k)]_{2, 1}|}{\sqrt{-\discr \frakX_i (x_k)}}.
	\end{equation}
	Since
	\begin{align*}
		\frac{|[\frakX_i (x_k)]_{2, 1}|}
		{\sqrt{-\discr \frakX_i (x_k)}}
		&=
		\frac{\tilde{a}_k |[\frakX_i (x_k)]_{2, 1}|}{\tilde{a}_k \sqrt{-\discr \frakX_i (x_k)}} \\
		&=
		\frac{\big| \big[ \tilde{a}_k \big( \frakX_i (x_k) - \sigma \Id \big) \big]_{2, 1} \big|}
		{\sqrt{- \discr \big( \tilde{a}_k \big( \frakX_i (x_k) - \sigma \Id \big) \big)}},
	\end{align*}
	by \eqref{eq:6}, we get
	\[
		\lim_{k \to \infty}
		\frac{|[\frakX_i (x_k)]_{2, 1}|}
		{\sqrt{-\discr \frakX_i (x_k)}}
		=
		\frac{1}{\alpha_{i-1}} \frac{|[\calR_i(x)]_{2, 1}|}{\sqrt{-\discr{\calR_i(x)}}}.
	\]
	Thus, by \eqref{eq:7} we arrive at
	\[
		\lim_{k \to \infty} \omega'(x_k) = 
		\frac{1}{\pi N} \sum_{i=0}^{N-1} 
		\frac{1}{\alpha_{i-1}} \frac{|[\calR_i(x)]_{2, 1}|}{\sqrt{-\discr{\calR_i(x)}}}.
	\]
	Since $\omega'$ is continuous at $0$, the conclusion follows.
\end{proof}

In the following theorem we show that there is a simple expression for $\upsilon$.
\begin{theorem} 
	\label{thm:7}
	Let $N$ be a positive integer and $\sigma \in \{-1, 1\}$. Let $(a_n : n \in \NN_0)$ and $(b_n : n \in \NN_0)$ be 
	$N$-periodically modulated Jacobi parameters
	so that $\frakX_0(0) = \sigma \Id$. Suppose that for each $i \in \{ 0, 1, \ldots, N-1 \}$ the limit
	\[
		\calR_i = \lim_{n \to \infty} a_{(n+1)N+i-1}(X_{nN+i} - \sigma \Id).
	\]	
	exists.	If
	\[
		\lim_{n \to \infty} (a_{n+N} - a_n) = 0,
	\]
	then for all $x \in \Lambda$ and $i \in \{ 0, 1, \ldots, N-1 \}$,
	\begin{align}
		\label{eq:65}
		\upsilon(x) &= 
		\lim_{n \to \infty} 
		\frac{a_{nN+i}}{\alpha_i} \frac{\big|\tr R_{nN+i}'(x)\big|}{N \pi \sqrt{-\discr R_{nN+i}(x)}} \\
		\label{eq:66}
		&=
		\frac{1}{4 N \pi \alpha_{i-1}}
		\frac{\big| (\discr \calR_i)'(x) \big|}{\sqrt{-\discr \calR_i(x)}}.
	\end{align}
\end{theorem}
\begin{proof}
	In the proof we use the truncated sequences $(a^L_n : n \in \NN_0)$ and $(b^L_n : n \in \NN_0)$
	defined by formulas \eqref{eq:42a} and \eqref{eq:42b}, respectively. Let $X_n^L$, $R_n$ and $R^L_n$ be defined in
	\eqref{eq:43} and \eqref{prop:1:eq:1}.
	\begin{claim} 
		\label{claim:3}
		\[
			\lim_{L \to \infty} \frac{a_L}{\alpha_L} \tr \Big( R_L - R^L_{N+L} \Big)' = 0.
		\]
	\end{claim}
	Observe that
	\[
		\Id - B_L^{-1}(x)
		\begin{pmatrix}
			0 & 1 \\
			-\frac{a_{L+N-1}}{a_L} & \frac{x-b_L}{a_L}
		\end{pmatrix}
		=
		\bigg(1-\frac{a_{L+N-1}}{a_{L-1}} \bigg)
		\begin{pmatrix}
			1 & 0 \\
			0 & 0
		\end{pmatrix},
	\]
	thus
	\begin{align*}
		X_L(x) - X_{L+N}^L(x) 
		&= X_L(x) 
		\bigg(
		\Id - B_L^{-1}(x) 
		\begin{pmatrix}
			0 & 1 \\
			-\frac{a_{L+N-1}}{a_L} & \frac{x-b_L}{a_L}
		\end{pmatrix}
		\bigg) \\
		&=
		\bigg(1 - \frac{a_{L+N-1}}{a_{L-1}}\bigg)
		X_L(x) 
		\begin{pmatrix}
			1 & 0 \\
			0 & 0
		\end{pmatrix}.
	\end{align*}
	Hence,
	\begin{align*}
		R_L(x) - R_{L+N}^L(x) &= a_{L+N-1} \big( X_L(x) - X_{L+N}^L(x) \big) \\
		&=
		a_{L+N-1} \bigg(1 - \frac{a_{L+N-1}}{a_{L-1}} \bigg) 
		X_L(x) 
		\begin{pmatrix}
			1 & 0 \\
			0 & 0
		\end{pmatrix},
	\end{align*}
	which, by \cite[Proposition 3]{PeriodicIII}, leads to
	\[
		\tr \Big( R_L - R_{L+N}^L \Big) = \frac{a_{L+N-1}}{a_{L-1}}
		\big( a_{L-1} - a_{L+N-1} \big) \bigg(-\frac{a_{L-1}}{a_L} p_{N-2}^{[L+1]} \bigg).
	\]
	By \cite[Proposition 3.9]{ChristoffelI}, for each $i \in \{0,1, \ldots, N-1\}$, we have
	\[
		\lim_{k \to \infty} \frac{a_{kN+i+1}}{\alpha_{i+1}} \big(p_{N-2}^{[kN+i+1]}\big)'(x) = 
		\big(\mathfrak{p}^{[i+1]}_{N-2} \big)'(0),
	\]
	thus 
	\[
		\lim_{k \to \infty} \frac{a_{kN+i}}{\alpha_i} \tr \Big( R_{kN+i} - R^{kN+i}_{kN+N+i} \Big)' =	0,
	\]
	which completes the proof of the claim.

	Next, we consider matrices $\frakX_m^L$ and $\calR^L_m$ defined for $m \in \NN_0$ as
	\[
		\frakX_m^L = X^L_{L+N+m}, \qquad \calR^L_m = R^L_{L+N+m}.
	\]
	Clearly, both sequences $\big( \frakX_m^L : m \in \NN_0 \big)$ and $\big( \calR^L_m : m \in \NN_0 \big)$
	are $N$-periodic. Let
	\[
		\Lambda_L = \big\{ x \in \RR : \discr R^L_{L+N}(x) < 0 \big\}.
	\]
	Since
	\[
		\discr X^L_{L+N} = \frac{1}{a_{L+N-1}^2} \discr R^L_{L+N},
	\]
	in view of \cite[formula (3.2)]{ChristoffelI}, for $x \in \Lambda_L$ we have
	\[
		\frac{| ( \tr \frakX^L_0 )'(x) |}{\sqrt{- \discr \frakX^L_0(x)}} 
		=
		\sum_{j=0}^{N-1} 
		\frac{1}{a_{L+j}} 
		\frac{|[\frakX^L_{j+1}(x)]_{2,1}|}{\sqrt{-\discr \frakX^L_{j+1}(x)}}.
	\]
	Hence,
	\begin{equation}
		\label{eq:1}
		\frac{| ( \tr \calR^L_0)'(x) |}{\sqrt{- \discr \calR^L_0(x)}} 
		=
		\sum_{j=0}^{N-1} 
		\frac{1}{a_{L+j}} 
		\frac{|[\calR^L_{j+1}(x)]_{2, 1}|}{\sqrt{-\discr \calR^L_{j+1}(x)}}.
	\end{equation}
	Let us now consider $x \in \Lambda$. By Proposition~\ref{prop:1}, $x \in \Lambda_L$ for sufficiently large $L$.
	Let $L = nN+i$, for $n \in \NN_0$ and $i \in \{0, 1, \ldots, N-1\}$. By Claim~\ref{claim:3}, we get
	\begin{equation}
		\label{eq:2}
		\lim_{n \to \infty}
		\frac{a_{nN+i}}{\alpha_i}
		\frac{| ( \tr R_{nN+i} )'(x) |}{\sqrt{- \discr R_{nN+i}(x)}}
		=
		\lim_{n \to \infty}
		\frac{a_{nN+i}}{\alpha_i}
		\frac{| ( \tr \calR^{nN+i}_0 )'(x) |}{\sqrt{- \discr \calR^{nN+i}_0(x)}}.
	\end{equation}
	Now we need the following statement.
	\begin{claim}
		\label{clm:5}
		\begin{equation}
			\label{eq:64}
			\lim_{n \to \infty} \calR^{nN+i}_j = \calR_{i+j}.
		\end{equation}
	\end{claim}
	First, let us see that \eqref{eq:64} together with \eqref{eq:2} and \eqref{eq:1} give
	\begin{align*}
		\lim_{n \to \infty}
		\frac{a_{nN+i}}{\alpha_i}
		\frac{| ( \tr R_{nN+i} )'(x) |}{\sqrt{- \discr R_{nN+i}(x)}}
		&=
		\sum_{j=0}^{N-1} 
		\frac{1}{\alpha_{i+j}} 
		\frac{|[\calR_{i+j+1}(x)]_{2, 1}|}{\sqrt{-\discr \calR_{i+j+1}(x)}}  \\
		&=
		\sum_{j=0}^{N-1} 
		\frac{1}{\alpha_{j-1}} 
		\frac{|[\calR_{j}(x)]_{2, 1}|}{\sqrt{-\discr \calR_{j}(x)}},
	\end{align*}
	proving \eqref{eq:65}.

	Therefore, it remains to prove Claim \ref{clm:5}. Observe that for $i' \in \{0, 1, \ldots, N-1 \}$,
	\[
		\calR_{i'}^L = R^L_{L+N+i'} =
		\frac{a_{L+i'-1}}{a_{L+N-1}} 
		\bigg( \prod_{k=L+N}^{L+N+i'-1} B^L_k \bigg)
		R^L_{L+N}
		\bigg( \prod_{k=L+N}^{L+N+i'-1} B^L_k \bigg)^{-1}.
	\]
	Hence, by Proposition~\ref{prop:1}, we get
	\begin{equation} 
		\label{eq:63}
		\lim_{n \to \infty} \calR_{i'}^{nN+i} (x)
		=
		\frac{\alpha_{i+i'-1}}{\alpha_{i-1}}
		\bigg( \prod_{k=i}^{i+i'-1} \frakB_k(0) \bigg)
		\calR_i(x)
		\bigg( \prod_{k=i}^{i+i'-1} \frakB_k(0) \bigg)^{-1}.
	\end{equation}
	Thus, by repeated application of Proposition~\ref{prop:5} to the right-hand side of \eqref{eq:63}, we arrive at \eqref{eq:64}.

	We now turn to proving \eqref{eq:66}. Since
	\[
		\discr \frakX_0^L = \frac{1}{a_{L+N-1}^2} \discr \calR^L_0,
	\]
	we easily get
	\begin{equation}
		\label{eq:56}
		( \discr \frakX^L_0 )' = \frac{1}{a_{L+N-1}^2} ( \discr \calR_0^L )'.
	\end{equation}
	On the other hand, we have $\det \frakX^L_0 \equiv 1$, thus
	\begin{equation}
		\label{eq:92}
		( \discr \frakX^L_0 )' = 2 \tr \frakX^L_0 ( \tr \frakX^L_0 )'.
	\end{equation}
	Moreover,
	\begin{equation}
		\label{eq:59}
		\tr \frakX^L_0 = 2 \sigma + \frac{1}{a_{L+N-1}} \tr \calR^L_0,
	\end{equation}
	thus for any $x \in \Lambda$ there is $L_x > 0$ such that for all $L \geq L_x$ one has $|\tr \frakX_0^L(x)| > 1$.
	Therefore, by \eqref{eq:92} and \eqref{eq:56} , we obtain
	\begin{align*}
		( \tr \frakX^L_0 )'(x) 
		&= \frac{1}{2 \tr \frakX^L_0(x) } ( \discr \frakX^L_0 )'(x) \\
		&= \frac{1}{a_{L+N-1}^2} \cdot  \frac{1}{2 \tr \frakX^L_0(x) } ( \discr \calR_0^L )'(x),
	\end{align*}
	which together with \eqref{eq:59} gives
	\[
		a_L
		\frac{(\tr \calR^L_0)'(x)}{\sqrt{-\discr \calR_0^L (x)}}
		=
		\frac{a_L}{a_{L+N-1}} \cdot \frac{1}{2 \tr \frakX^L_0(x) } 
		\cdot
		\frac{( \discr \calR^L_0 )'(x)}{\sqrt{-\discr \calR^L_0(x)}}.
	\]
	Consequently, we get
	\[
		\lim_{n \to \infty} 
		\frac{a_{nN+i}}{\alpha_i}
		\cdot
		\frac{| ( \tr \calR^L_0 )'(x) |}{\sqrt{-\discr \calR_0^L(x)}} 
		=
		\frac{1}{4 \alpha_{i-1}} \cdot \frac{|(\discr \calR_i)'(x)|}{\sqrt{-\discr \calR_i(x)}},
	\]
	and the conclusion follows by \eqref{eq:2} and \eqref{eq:65}.
\end{proof}

\begin{proposition}
	For all $x \in \Lambda$, $\upsilon(x) > 0$.
\end{proposition}
\begin{proof}
	Suppose, contrary to our claim, that $\upsilon(x_0) = 0$ for some $x_0 \in \Lambda$. Then by \eqref{eq:33},
	\[
		[\calR_0(x_0)]_{2,1} = 0,
	\]
	and consequently,
	\[
		\discr \calR_0(x_0) = \Big( [\calR_0(x_0)]_{1,1} - [\calR_0(x_0)]_{2,2} \Big)^2 \geq 0,
	\]
	which in view of \eqref{eq:32} leads to contradiction.
\end{proof}

The following two examples demonstrates that the assumption \eqref{thm:2:eq:1} is necessary for the conclusion 
\eqref{thm:2:eq:2} to hold.
\begin{example} 
	\label{ex:3}
	Let $N=2$. Suppose that $A$ has $N$-periodically modulated entries corresponding to
	\[
		\alpha_n \equiv 1, \qquad\text{and}\qquad \beta_n \equiv 0.
	\]
	Assume that \eqref{eq:30} is satisfied with
	\[
		s_n = (-1)^n, \qquad\text{and}\qquad
		z_n \equiv 0.
	\]
	Then, by \cite[Proposition 9]{PeriodicIII}, we have
	\[
		\calR_0(x) = 
		\begin{pmatrix}
			-1 & x \\
			-x & 1
		\end{pmatrix},
		\qquad\text{and}\qquad
		\calR_1(x) = 
		\begin{pmatrix}
			1 & x \\
			-x & -1
		\end{pmatrix}
	\]
	Hence, $\Lambda = \RR \setminus [-1, 1]$, and, by \eqref{eq:33},
	\[
		\upsilon(x) = 
		\frac{|x|}{\pi \sqrt{4 x^2 - 4}}, \qquad x \in \Lambda
	\]
	which agrees with the formula \eqref{eq:66} (in view of Proposition \ref{prop:2}, the hypotheses of
	Theorem~\ref{thm:7} are satisfied).
\end{example}

\begin{example} 
	\label{ex:4}
	Let $N=2$. Suppose that $A$ has $N$-periodically modulated entries corresponding to
	\[
		\alpha_n \equiv 1, \qquad\text{and}\qquad
		\beta_n \equiv 0.
	\]
	Assume that \eqref{eq:30} is satisfied with
	\[
		s_n \equiv 0, \qquad\text{and}\qquad
		z_n = -\frac{(-1)^n+1}{2}.
	\]
	Then, by \cite[Proposition 9]{PeriodicIII}, we have
	\[
		\calR_0(x) = 
		\begin{pmatrix}
			0 & x \\
			-x + 1 & 0
		\end{pmatrix},
		\qquad\text{and}\qquad
		\calR_1(x) = 
		\begin{pmatrix}
			0 & x-1 \\
			-x & 0
		\end{pmatrix}
	\]
	Hence, $\Lambda = \RR \setminus [0, 1]$, and 
	\[
		\upsilon(x) = 
		\frac{|x| + |x-1|}{2 \pi \sqrt{4x^2 - 4x}} = \frac{|2x-1|}{2 \pi \sqrt{4x^2 - 4x}}, \qquad x \in \Lambda
	\]
	which agrees with the formula \eqref{eq:66} (in view of Proposition \ref{prop:2}, the hypotheses of
	Theorem~\ref{thm:7} are satisfied).
\end{example}

\section{Christoffel functions for periodic modulations}
\label{sec:7}
\begin{theorem}
	\label{thm:9}
	Let $N$ be a positive integer and $\sigma \in \{-1, 1\}$. Let $(a_n : n \in \NN_0)$ and $(b_n : n \in \NN_0)$ be 
	$N$-periodically modulated Jacobi parameters so that $\frakX_0(0) = \sigma \Id$. Suppose that 
	\begin{equation}
		\label{eq:79}
		\lim_{n \to \infty} (a_{n+N} - a_{n}) = 0.
	\end{equation}
	Set
	\[
		R_n = a_{n+N-1} \big( X_n - \sigma \Id \big).
	\]
	Let $K \subset \Lambda$ be a compact interval with non-empty interior, where $\Lambda$ is defined in \eqref{eq:32}.
	Suppose that
	\begin{equation} \label{eq:80}
		(X_n : n \in \NN), (R_n : n \in \NN) \in \calD_1^N \big( K, \GL(2, \RR) \big)
	\end{equation}
	Then
	\[
		\lim_{n \to \infty} \frac{1}{\rho_n} K_n(x, x) = \frac{\upsilon(x)}{\mu'(x)}
	\]
	uniformly with respect to $x \in K$, where $\upsilon$ is given by \eqref{eq:33} and
	\[
		\rho_n = \sum_{j = 0}^n \frac{\alpha_j}{a_j}.
	\]
\end{theorem}
\begin{proof}
	Let us fix a compact interval $K$ with non-empty interior contained in $\Lambda$. Consider
	$i \in \{0, 1, \ldots, N-1\}$. By \eqref{eq:79}, there is $c > 0$ such that for all $k \in \NN_0$,
	\[
		a_{kN+i-1} = a_{i-1} + \sum_{j = 1}^k \big( a_{jN+i-1} - a_{(j-1)N+i-1} \big) \leq c(k+1),
	\]
	hence
	\[
		\lim_{k \to \infty} \rho_{i-1; k} = \infty.
	\]
	Thus, Theorem \ref{thm:6} easily leads to
	\begin{equation}
		\label{eq:81}
		K_{i; n}(x, x) = \frac{1}{\pi \mu'(x)} \frac{|[\calR_i(x)]_{2,1}|}{\sqrt{-\discr \calR_i(x)}} \rho_{i-1; n}
		+E_{i; n}(x)
	\end{equation}
	where
	\[
		\lim_{n \to \infty} 
		\frac{1}{\rho_{i-1; n}} \sup_{x \in K} |E_{i; n}(x)| = 0.
	\]
	Next, by \cite[Proposition 3.7]{ChristoffelI}, for $n, n' \in \NN_0$,
	\[
		\lim_{j \to \infty} \frac{a_{jN+n'}}{a_{jN+n}} = \frac{\alpha_{n'}}{\alpha_n},
	\]
	thus, by the Stolz--Ces\'aro theorem, for each $i, i' \in \{0, 1, \ldots, N-1\}$,
	\begin{align}
		\nonumber
		\lim_{j \to \infty} \frac{\rho_{i'; j}}{\rho_{jN+i}} 
		&= 
		\lim_{j \to \infty} \frac{\frac{1}{a_{jN+i'}}}{\sum_{k = 1}^N \frac{\alpha_{i+k}}{a_{jN+i+k}}} \\
		\label{eq:78}
		&=
		\frac{1}{N \alpha_{i'}}.
	\end{align}
	Let us now consider $n = kN+i$ where $i \in \{0, 1, \ldots, N-1\}$. We write
	\[
		K_{kN+i}(x, x) 
		= \sum_{i' = 0}^{N-1} K_{i'; k}(x, x) + \sum_{i' = i+1}^{N-1} (K_{i'; k-1}(x, x) - K_{i'; k}(x, x)).
	\]
	Observe that by Theorem \ref{thm:5},
	\[
		\sup_{x \in K} \big|K_{i'; k-1}(x, x) - K_{i'; k}(x, x)\big|
		= \sup_{x \in K} p_{kN+i'}^2 (x) \leq c,
	\]
	hence, by \eqref{eq:81},
	\[
		\lim_{k \to \infty} \frac{1}{\rho_{kN+i}} K_{kN+i}(x, x)
		=
		\sum_{i' = 0}^{N-1}
		\frac{1}{\pi \mu'(x)} \frac{|[\calR_{i'}(x)]_{2,1}|}{\sqrt{-\discr \calR_{i'}(x)}}
		\cdot \lim_{k \to \infty} \frac{\rho_{i'-1; k}}{\rho_{kN+i}}.
	\]
	Using now \eqref{eq:78} and \eqref{eq:33}, we obtain
	\begin{align*}
		\lim_{k \to \infty} \frac{1}{\rho_{kN+i}} K_{kN+i}(x, x)
		&=
		\frac{1}{\mu'(x)} \cdot 
		\frac{1}{N \pi}
		\sum_{i' = 0}^{N-1}
		\frac{|[\calR_{i'}(x)]_{2,1}|}{\sqrt{-\discr \calR_{i'}(x)}} 
		\cdot \frac{1}{\alpha_{i'-1}} \\
		&=
		\frac{\upsilon(x)}{\mu'(x)},
	\end{align*}
	which completes the proof.
\end{proof}

\begin{remark} \label{rem:1}
If one assumes that
\[
	\bigg( \frac{\alpha_{n-1}}{\alpha_n} a_n - a_{n-1} : n \in \NN \bigg),
	\bigg( \frac{\beta_n}{\alpha_n} a_n - b_n : n \in \NN \bigg), 
	\bigg( \frac{1}{a_n} : n \in \NN \bigg) \in \calD_1^N(\RR),
\]
then condition \eqref{eq:80} is satisfied for any compact $K \subset \RR$, see \cite[Proposition 9]{PeriodicIII}.
\end{remark}

\subsection{Applications to Ignjatović conjecture}
In this section we show how Theorem \ref{thm:9} leads to the conjecture due to Ignjatović
\cite[Conjecture 1]{Ignjatovic2016}.
\begin{conjecture}[Ignjatović, 2016]
	\label{conj:1}
	Suppose that
	\begin{itemize}
		\item[($\calC_1$)]
		$\begin{aligned}[b]
			\lim_{n \to \infty} a_n = \infty
		\end{aligned}$;
		\item[($\calC_2$)]
		$\begin{aligned}[b]
			\lim_{n \to \infty} \Delta a_n = 0
		\end{aligned}$;
		\item[($\calC_3$)]
		There exist $n_0, m_0$ such that $a_{n+m} > a_n$ holds for all $n \geq n_0$ and all $m \geq m_0$;
		\item[($\calC_4$)]
		$\begin{aligned}[b]
		\sum_{n=0}^\infty \frac{1}{a_n} = \infty
		\end{aligned}$;
		\item[($\calC_5$)]
		There exists $\kappa > 1$ such that $\sum_{n=0}^\infty \frac{1}{a_n^\kappa} < \infty$;
		\item[($\calC_6$)]
		$\begin{aligned}[b]
		\sum_{n=0}^\infty \frac{|\Delta a_n|}{a_n^2} < \infty
		\end{aligned}$;
		\item[($\calC_7$)]
		$\begin{aligned}[b]
			\sum_{n=0}^\infty \frac{\big|\Delta^2 a_n \big|}{a_n} < \infty
		\end{aligned}$.
	\end{itemize}
	If
	\[
		-2 < \lim_{n \to \infty} \frac{b_n}{a_n} < 2,
	\]
	then for any $x \in \RR$, the limit
	\[
		\lim_{n \to \infty} \bigg( \sum_{j=0}^n \frac{1}{a_j} \bigg)^{-1} \sum_{j=0}^n p_j^2(x)
	\]
	exists and is positive.
\end{conjecture}
Our results entail the following corollary.
\begin{corollary}
	\label{cor:2}
	Let $N$ be a positive integer. Suppose that
	\[
		\lim_{n \to \infty} \frac{a_{n-1}}{a_n} = 1, \qquad
		\lim_{n \to \infty} \frac{b_n}{a_n} = q, \qquad
		\lim_{n \to \infty} a_n = \infty, \qquad
		\lim_{n \to \infty} (a_{n+N} - a_n) = 0,
	\]
	for some
	\begin{equation} 
		\label{eq:67}
		q \in \big\{2 \cos(j \tfrac{\pi}{N}) : j = 1, 2, \ldots, N-1 \big\}.
	\end{equation}
	If
	\[
		\big( a_{n} - a_{n-1} : n \in \NN \big), 
		\big( b_n - q a_n : n \in \NN \big), 
		\bigg( \frac{1}{a_{n}} : n \in \NN \bigg) \in \calD_1^N(\RR),
	\]
	then
	\[
		\lim_{n \to \infty} 
		\bigg( \sum_{j=0}^n \frac{1}{a_j} \bigg)^{-1} \sum_{j=0}^n p_j^2(x) 
		= \frac{\upsilon(x)}{\mu'(x)},
	\]
	locally uniformly with respect to $x \in \Lambda$, where $\Lambda$ and $\upsilon$ are defined 
	in \eqref{eq:32} and \eqref{eq:33}, respectively.
\end{corollary}
\begin{proof}
	Let 
	\[
		\alpha_n \equiv 1, \quad \beta_n \equiv q.
	\]
	Observe that 
	\[
		\frakX_0(0) =
		\begin{pmatrix}
			0 & 1 \\
			-1 & -q
		\end{pmatrix}^N =
		\begin{pmatrix}
			0 & 1 \\
			-1 & \frac{-q/2 - 0}{1/2}
		\end{pmatrix}^N.
	\]
	Hence, by \cite[Lemma 3.2]{ChristoffelI}
	\[
		\frakX_0(0) = 
		\begin{pmatrix}
			-U_{N-2}(-\tfrac{q}{2}) & U_{N-1}(-\tfrac{q}{2}) \\
			-U_{N-1}(-\tfrac{q}{2}) & U_{N}(-\tfrac{q}{2})
		\end{pmatrix}
	\]
	where $(U_n : n \in \NN_0)$ is the sequence of Chebyshev polynomials of the second kind defined as
	\[
		U_n(x) = \frac{\sin \big( (n+1) \arccos(x) \big)}{\sin \big( \arccos(x) \big)}, \qquad x \in (-1,1)
	\]
	from which we readily derive
	\[
		\frakX_0(0) = \sigma \Id
	\]
	with $\sigma = (-1)^{N+j}$. Now, in view of Remark~\ref{rem:1}, the conclusion is a consequence of 
	Theorem~\ref{thm:9}.
\end{proof}

By using a different method the conclusion of Corollary~\ref{cor:2} for $q=0$, $s_n \equiv 0$ and
$z_n \equiv 0$ (cf. formula \eqref{eq:30}) has been proven in \cite[Corollary 3]{PeriodicII}. 

Let us recall that in \cite[Corollary 4.16]{ChristoffelI} there was considered the case when 
\eqref{eq:67} is \emph{not} satisfied. In fact, under some hypotheses it was shown that
\begin{equation} 
	\label{eq:68}
	\lim_{n \to \infty} 
	\bigg( \sum_{j=0}^n \frac{1}{a_j} \bigg)^{-1} \sum_{j=0}^n p_j^2(x) 
	= \frac{1}{\pi \sqrt{4 - q^2}} \frac{1}{\mu'(x)}
\end{equation}
locally uniformly with respect to $x \in \RR$. Let us stress that in the setup of Corollary~\ref{cor:1}
it is still possible to have \eqref{eq:68} locally uniformly with respect to $x \in \RR \setminus \{ 0 \}$, see Theorem~\ref{thm:2}. Nevertheless, as the next example demonstrates, it is \emph{not} 
the case that one always obtains \eqref{eq:68} provided that the left hand side exists.

\begin{example}
	Let
	\[
		a_n = \sqrt{n+1}, \qquad\text{and}\qquad
		b_n = \frac{1 - (-1)^n}{2}.
	\]
	Then
	\[
		\lim_{n \to \infty} 
		\bigg( \sum_{j=0}^n \frac{1}{a_j} \bigg)^{-1} \sum_{j=0}^n p_j^2(x) 
		= 
		\frac{|2x - 1|}{4 \pi \sqrt{x(x-1)}}
		\frac{1}{\mu'(x)}
	\]
	locally uniformly with respect to $x \in \RR \setminus [0,1]$. Indeed, since the hypothesis of Corollary \ref{cor:2}
	are satisfied for $N = 2$ and $q = 0$, the conclusion follows from Example \ref{ex:4}. 
\end{example}

\section{Universality limits of Christoffel--Darboux kernel}
\label{sec:8}

\begin{proposition}
	\label{prop:4}
	Let $N$ be a positive integer and $\sigma \in \{-1, 1\}$. Let $(a_n : n \in \NN_0)$ and $(b_n : n \in \NN_0)$ be 
	$N$-periodically modulated Jacobi parameters so that $\frakX_0(0) = \sigma \Id$. Suppose that for each 
	$i \in \{0, 1, \ldots, N-1 \}$ the limit
	\[
		\calR_i = \lim_{k \to \infty} a_{(k+1)N+i-1}(X_{kN+i} - \sigma \Id).
	\]	
	exists. Let
	\begin{equation}
		\label{eq:82}
		\theta_n(x) = \arccos\bigg(\frac{\tr X_n(x)}{2 \sqrt{\det X_n(x)}} \bigg),
	\end{equation}
	Then for each $i \in \{0, 1, \ldots, N-1\}$,
	\begin{equation} 
		\label{eq:72}
		\lim_{k \to \infty} \frac{a_{kN+i}}{\alpha_{i}} \big| \theta_{kN+i}'(x) \big| = N \pi \upsilon(x),
	\end{equation}
	and
	\begin{equation} 
		\label{eq:73}
		\lim_{k \to \infty}
		\frac{a_{kN+i}}{\alpha_{i}} \theta_{kN+i}''(x) = 
		-\alpha_{i-1} \frac{\tr \frakX''_i(0)}{\sqrt{-\discr \calR_i(x)}} 
		-\alpha_{i-1} \frac{2 \sigma \big( N \pi \upsilon(x) \big)^2}{\big( -\discr \calR_i(x) \big)^{3/2}}
	\end{equation}
	locally uniformly with respect to $x \in \Lambda$, where $\Lambda$ and $\upsilon$ are given by \eqref{eq:32} and \eqref{eq:33}, respectively.
\end{proposition}
\begin{proof}
	Let us fix $i \in \{ 0, 1, \ldots, N-1 \}$. Since
	\begin{equation}
		\label{eq:107}
		\det X_{kN+i}(x) = \frac{a_{kN+i-1}}{a_{(k+1)N+i-1}},
	\end{equation}
	we conclude that
	\[
		\lim_{k \to \infty} \det X_{kN+i}(x) = \det \big( \sigma \Id \big) = 1.
	\]
	The chain rule applied to \eqref{eq:82} leads to
	\begin{align} 
		\nonumber
		\theta_{kN+i}'(x) 
		&=
		-\Bigg(4 - \bigg( \frac{\tr X_{kN+i}(x)}{\sqrt{\det X_{kN+i}(x)}} \bigg)^2 \Bigg)^{-1/2}
		\frac{\tr X_{kN+i}'(x)}{\sqrt{\det X_{kN+i}(x)}} \\
		\label{eq:83}
		&= 
		-\frac{\tr X_{kN+i}'(x)}{\sqrt{-\discr X_{kN+i}(x)}},
	\end{align}
	thus
	\[
		\theta_{kN+i}'(x) = -\frac{\tr R_{kN+i}'(x)}{\sqrt{-\discr R_{kN+i}(x)}}
	\]
	where we have set
	\[
		R_n = a_{n+N-1}(X_n - \sigma \Id).
	\]
	Hence, the formula \eqref{eq:72} is a consequence of \eqref{eq:65}.

	Next, by taking derivative of \eqref{eq:83}, we obtain
	\begin{equation} 
		\label{eq:77}
		\theta_{kN+i}''(x) 
		= 
		-\frac{\tr X_{kN+i}''(x)}{\sqrt{-\discr X_{kN+i}(x)}} - 
		\frac{\big(\tr X_{kN+i}'(x)\big)^2 \tr X_{kN+i}(x)}{\big( -\discr X_{kN+i}(x) \big)^{3/2}}.
	\end{equation}
	Since
	\[
		\frac{a_{kN+i}}{\alpha_i}
		\frac{\tr X_{kN+i}''(x)}{\sqrt{-\discr X_{kN+i}(x)}} = 
		\alpha_i \frac{a_{(k+1)N+i-1}}{a_{kN+i}}
		\frac{\big( \tfrac{a_{kN+i}}{\alpha_i}\tr X_{kN+i}''(x)\big)^2}{\sqrt{-\discr R_{kN+i}(x)}},
	\]
	by \cite[Corollary 3.10]{ChristoffelI}, we get
	\begin{equation} 
		\label{eq:75}
		\lim_{k \to \infty} 
		\frac{a_{kN+i}}{\alpha_i} 
		\frac{\tr X_{kN+i}''(x)}{\sqrt{-\discr X_{kN+i}(x)}} 
		=
		\alpha_{i-1} \frac{\tr \frakX''_i(0)}{\sqrt{-\discr \calR_i(x)}}.
	\end{equation}
	Similarly,
	\[
		\frac{a_{kN+i}}{\alpha_i}
		\frac{\big(\tr X_{kN+i}'(x)\big)^2 \tr X_{kN+i}(x)}{\big( -\discr X_{kN+i}(x) \big)^{3/2}}
		=
		\alpha_i \frac{a_{(k+1)N+i-1}}{a_{kN+i}}
		\frac{\big( \tfrac{a_{kN+i}}{\alpha_i} \tr R_{kN+i}'(x)\big)^2 \tr X_{kN+i}(x)}
		{\big( -\discr R_{kN+i}(x) \big)^{3/2}},
	\]
	and, by \eqref{eq:65},
	\begin{equation} \label{eq:76}
		\lim_{k \to \infty}
		\frac{a_{kN+i}}{\alpha_i}
		\frac{\big(\tr X_{kN+i}'(x)\big)^2 \tr X_{kN+i}(x)}{\big( -\discr X_{kN+i}(x) \big)^{3/2}}
		=
		\alpha_{i-1}
		\frac{2 \sigma \big( N \pi \upsilon(x) \big)^2}
		{\big( -\discr \calR_{i}(x) \big)^{3/2}}.
	\end{equation}
	Finally, combining \eqref{eq:75} and \eqref{eq:76} with \eqref{eq:77} we obtain \eqref{eq:73}.
	This completes the proof.
\end{proof}

\begin{lemma}
	\label{lem:2}
	Let $(\gamma_k : k \in \NN_0)$ be a sequence of positive numbers such that
	\[
		\sum_{k = 0}^\infty \gamma_k = \infty, \qquad\text{and}\qquad
		\lim_{k \to \infty} \frac{\gamma_{k-1}}{\gamma_k} = 1.
	\]
	Assume that $(\xi_k : k \in \NN_0)$ is a sequence of continuous functions on some open subset $U \subset \RR^d$,
	with values in $(0, 2\pi)$. Suppose that there is $\psi: U \rightarrow (0, 2\pi)$ such that
	\[
		\lim_{n \to \infty} \frac{\xi_n(x)}{\gamma_n} = \psi(x)
	\]
	locally uniformly with respect to $x \in U$. Let $(r_n : n \in \NN)$ be a sequence of positive numbers such that
	\[
		\lim_{n \to \infty} r_n = \infty.
	\]
	For $x \in U$ and $a, b \in \RR$, we set
	\[
		x_n = x + \frac{a}{r_n}, \qquad y_n = x + \frac{b}{r_n}.
	\]
	Then for each compact subset $K \subset U$, $L > 0$ and any function $\sigma: U \rightarrow \RR$,
	\[
		\lim_{n \to \infty} \sum_{k = 0}^n \frac{\gamma_k}{\sum_{j = 0}^n \gamma_j}
		\cos\Big(\sum_{j = 0}^k \xi_j(x_n) + \xi_j(y_n) + \sigma(x_n) + \sigma(y_n)\Big) = 0
	\]
	uniformly with respect to $x \in K$, and $a, b \in [-L, L]$.
\end{lemma}
\begin{proof}
	Let us fix a compact set $K$, and let $N \in \NN$ be such that $r_n \geq R$, for $n \geq N$. For
	$(x, a, b) \in U \times (-2 L, 2L)^2$, we set
	\[
		\tilde{\xi}_0(x, a, b) = \xi_0\bigg(x+\frac{a}{R}\bigg) + \xi_0\bigg(x+ \frac{b}{R}\bigg)
		+\sigma\bigg(x+\frac{a}{R}\bigg) + \sigma\bigg(x+\frac{b}{R}\bigg),
	\]
	and
	\[
		\tilde{\xi}_j(x, a, b) = \xi_j\bigg(x+\frac{a}{R}\bigg) + \xi_j\bigg(x+ \frac{b}{R}\bigg),
		\quad\text{and}\quad
		\tilde{\psi}(x, a, b) = \psi\bigg(x+\frac{a}{R}\bigg) + \psi\bigg(x + \frac{b}{R}\bigg).
	\]
	Thus
	\begin{equation}
		\label{eq:87}
		\lim_{j \to \infty} \frac{1}{\gamma_j} \tilde{\xi}_j(x, a, b)= \tilde{\psi}(x, a, b).
	\end{equation}
	In view of Lemma \ref{lem:1} we obtain
	\[
		\lim_{n \to \infty}
		\sum_{k = 0}^n
		\frac{\gamma_k}{\sum_{j = 0}^n \gamma_j} 
		\cos\Big( \sum_{j = 0}^k \tilde{\xi}_j(x, a, b) \Big)
		=\lim_{n \to \infty}
		\frac{1}{\tilde{\Xi}_n(x, a, b)} \int_0^{\tilde{\Xi}_n(x, a, b)}  \cos(t) \ud t
	\]
	where
	\[
		\tilde{\Xi}_n(x, a, b) = \sum_{j = 0}^n \tilde{\xi}_j(x, a, b).
	\]
	By \eqref{eq:87}, there is $c > 0$ such that for all $x \in K$, and $a, b \in [-L, L]$,
	\[
		\tilde{\xi}_j(x, a, b) \geq c \gamma_j.
	\]
	Hence,
	\[
		\frac{1}{\tilde{\Xi}_n(x, a, b)} \int_0^{\tilde{\Xi}_n(x, a, b)} \cos(t) \ud t
		\leq
		\frac{1}{\tilde{\Xi}_n(x, a, b)},
	\]
	which implies that
	\[
		\lim_{n \to \infty}
	    \sum_{k = 0}^n
		\frac{\gamma_k}{\sum_{j = 0}^n \gamma_j}
		\cos\Big( \sum_{j = 0}^k \tilde{\xi}_j(x, a, b) \Big)
		=
		\lim_{n \to \infty} \frac{1}{\tilde{\Xi}_n(x, a, b)} \int_0^{\tilde{\Xi}_n(x, a, b)}  \cos(t) \ud t
		=
		0,
	\]
	uniformly with respect to $x \in K$, $a, b \in [-L, L]$.
\end{proof}

\begin{theorem}
	\label{thm:11}
	Assume that $(\xi_j : j \in \NN_0)$ is a sequence of $\calC^2(U)$ functions with values in $(0, 2\pi)$ such that
	for each compact set $K \subset U$ there are functions $\xi : U \rightarrow (0, \infty)$ and
	$\psi: U \rightarrow (0, \infty)$, and $c > 0$ so that
	\begin{enumerate}[(a)]
	\item 
		$\begin{aligned}[b] 
			\lim_{n \to \infty} \sup_{x \in K}{\big|\gamma_n^{-1} \cdot \xi_n(x) - \xi(x)\big|} = 0,
		\end{aligned}$
	\item
		$\begin{aligned}[b]
			\lim_{n \to \infty} \sup_{x \in K}{\big| \gamma_n^{-1} \cdot \xi_n'(x) - \psi(x) \big|} =0,
		\end{aligned}$
	\item
		\label{eq:89}
		$\begin{aligned}[b]
			\sup_{n \in \NN} \sup_{x \in K}{\big|\gamma_n^{-1} \cdot \xi_n''(x)\big|} \leq c,
		\end{aligned}$
	\end{enumerate}
	where $(\gamma_k : k \in \NN_0)$ is a sequence of positive numbers such that
	\[
		\sum_{k = 0}^\infty \gamma_k = \infty, \qquad\text{and}\qquad
		\lim_{k \to \infty} \frac{\gamma_{k-1}}{\gamma_k} = 1.
	\]
	For $x \in U$ and $a, b \in [-L, L]$, we set
	\[
		x_n = x + \frac{a}{\sum_{k = 0}^n \gamma_k}, \qquad
		y_n = x + \frac{b}{\sum_{k = 0}^n \gamma_k}.
	\]
	Then for any continuous function $\sigma: U \rightarrow \RR$,
	\begin{align*}
		\lim_{n \to \infty} \sum_{k = 0}^n \frac{\gamma_k}{\sum_{j = 0}^n \gamma_j}
		\sin\Big(\sum_{j = 0}^k \xi_j(x_n)+\sigma(x_n)\Big)
		\sin\Big(\sum_{j = 0}^k \xi_j(y_n)+\sigma(y_n)\Big)
		=
		\frac{1}{2} \sinc \big((b-a)\psi(x) \big)
	\end{align*}
	locally uniformly with respect to $x \in U$, and $a, b \in \RR$.
\end{theorem}
\begin{proof}
	We write
	\begin{align*}
		2 \cdot 
		\sin\Big(\sum_{j = 0}^k \xi_j(x) + \sigma(x) \Big) 
		\sin\Big(\sum_{j = 0}^k \xi_j(y) + \sigma(y) \Big)
		&=
		\cos\Big(\sum_{j = 0}^k\big(\xi_j(x) - \xi_j(y)\big) + \big(\sigma(x) - \sigma(y)\big)\Big) \\
		&\phantom{=}-
		\cos\Big(\sum_{j = 0}^k\big(\xi_j(x) + \xi_j(y)\big) + \big(\sigma(x) + \sigma(y)\big)\Big).
	\end{align*}
	By Lemma \ref{lem:2}, we conclude that
	\[
		\lim_{n \to \infty} 
		\sum_{k = 0}^n \frac{\gamma_k}{\sum_{j = 0}^n \gamma_j}
		\cos\Big(\sum_{j = 0}^k \big(\xi_j(x_n) + \xi_j(y_n)\big) + \big(\sigma(x_n) + \sigma(y_n)\big) \Big)
		=0.
	\]
	We next write 
	\begin{align*}
		\cos\Big(\sum_{j = 0}^k \big(\xi_j(x)-\xi_j(y)\big) + \big(\sigma(x)-\sigma(y)\big)\Big)
		&=
		\cos\Big(\sum_{j = 0}^k \big(\xi_j(x) - \xi_j(y) \big)\Big)\cos\big(\sigma(y) - \sigma(x)\big) \\
		&\phantom{=}-
		\sin\Big(\sum_{j = 0}^k \big(\xi_j(x) - \xi_j(y) \big)\Big)\sin\big(\sigma(y) - \sigma(x)\big),
	\end{align*}
	thus, it is enough to prove that
	\[
		\lim_{n \to \infty}
		\sum_{k = 0}^n \frac{\gamma_k}{\sum_{j = 0}^n \gamma_j}
		\cos\Big(\sum_{j = 0}^k \big(\xi_j(y_n) - \xi_j(x_n)\big)\Big) 
		=
		\sinc \big((b-a)\psi(x) \big)
	\]
	locally uniformly with respect to $x \in U$ and $a, b \in \RR$. Since (see, e.g., \cite[Claim 5.10]{ChristoffelI})
	\[
		\Big|
		\xi_j(y_n) - \xi_j(x_n) - (b-a) \xi_j'(x) \Big(\sum_{\ell = 0}^n \gamma_\ell\Big)^{-1}
		\Big|
		\leq
		c \Big(\sum_{\ell = 0}^n \gamma_\ell\Big)^{-2}
		\sup_{u \in K}{|\xi_j''(u)|},
	\]
	we obtain
	\begin{align*}
		&\Big|
		\cos \Big(\sum_{j = 0}^k \big(\xi_j(x_n) - \xi_j(y_n)\big)\Big)
		-
		\cos\Big((b-a)\Big(\sum_{\ell = 0}^n \gamma_\ell\Big)^{-1} \sum_{j = 0}^k \xi'_j(x) \Big)
		\Big|\\
		&\qquad\qquad\leq
		\sum_{j = 0}^k \Big|\xi_j(y_n) - \xi_j(x_n) - (b-a) \xi_j'(x) \Big(\sum_{j = 0}^n \gamma_\ell\Big)^{-1} \Big| \\
		&\qquad\qquad\leq
		c \Big(\sum_{\ell = 0}^n \gamma_\ell\Big)^{-2} \sum_{j = 0}^k \sup_{u \in K}{|\xi_j''(u)|}.
	\end{align*}
	Consequently,
	\begin{align*}
		&
		\bigg|
		\sum_{k = 0}^n \frac{\gamma_k}{\sum_{j = 0}^n \gamma_j}
		\bigg(\cos\Big(\sum_{j = 0}^k \xi_j(x_n)-\xi_j(y_n) \Big)
		-
		\cos\Big((b-a)\Big(\sum_{\ell = 0}^n \gamma_\ell \Big)^{-1} \sum_{j = 0}^k \xi'_j(x)\Big)
		\bigg)
		\bigg| \\
		&\qquad\qquad\leq
		c
		\Big( \sum_{\ell = 0}^n \gamma_\ell\Big)^{-3}
		\sum_{k = 0}^n \gamma_k \sum_{j = 0}^k \sup_{u \in K}{|\xi''_j(u)|}.
	\end{align*}
	In view of the Stolz--Ces\`aro theorem,
	\[
		\lim_{n \to \infty}	\frac{\gamma_n}{\sum_{j = 0}^n \gamma_j}
		=
		\lim_{n \to \infty} \frac{\gamma_n - \gamma_{n-1}}{\gamma_n} = 0,
	\]
	thus, by repeated application of the Stolz--Ces\`aro theorem we arrive at
	\begin{align*}
		\lim_{n \to \infty} \Big(\sum_{\ell = 0}^n \gamma_\ell \Big)^{-3} \sum_{k = 0}^n \gamma_k \sum_{j = 0}^k
		\sup_{u \in K}{|\xi''_j(u)|}
		&=
		\frac{1}{3} \lim_{n \to \infty} \Big(\sum_{\ell = 0}^n \gamma_\ell \Big)^{-2} 
		\sum_{j = 0}^n \sup_{u \in K}{|\xi''_j(u)|} \\
		&=
		\frac{1}{6} \lim_{n \to \infty} \gamma_n^{-1} \Big(\sum_{\ell = 0}^n \gamma_\ell \Big)^{-1} 
		\sup_{u \in K}{|\xi''_n(u)|} = 0
	\end{align*}
	where the last equality follows by \eqref{eq:89}. Now, our task is to show that
	\[
		\lim_{n \to \infty}
		\sum_{k = 0}^n \frac{\gamma_k}{\sum_{\ell = 0}^n \gamma_\ell} 
		\cos\Big((b-a) \Big(\sum_{j = 0}^n \gamma_\ell\Big)^{-1} \sum_{j = 0}^k \xi'_j(x)\Big)
		=
		\sinc\big((b-a)\psi(x) \big)
	\]
	locally uniformly with respect to $x \in U$ and $a, b \in \RR$. At this point we apply Lemma \ref{lem:3}
	to get
	\begin{align*}
		\lim_{n \to \infty}
		\sum_{k = 0}^n \frac{\gamma_k}{\sum_{\ell = 0}^n \gamma_\ell}
		\cos\Big((b-a) \Big(\sum_{j = 0}^n \gamma_\ell\Big)^{-1} \sum_{j = 0}^k \xi'_j(x)\Big)
		&=
		\frac{1}{\psi(x)}
		\int_0^{\psi(x)} \cos((b-a) t) \ud t \\
		&=
		\sinc\big((b-a) \psi(x)\big),
	\end{align*}
	and the theorem follows.
\end{proof}

\begin{theorem}
	\label{thm:10}
	Let $N$ be a positive integer and $\sigma \in \{-1, 1\}$. Let $(a_n : n \in \NN_0)$ and $(b_n : n \in \NN_0)$ be 
	$N$-periodically modulated Jacobi parameters so that $\frakX_0(0) = \sigma \Id$. Suppose that 
	\[
		\lim_{n \to \infty} (a_{n+N} - a_{n}) = 0
	\]
	and for each $i \in \{0, 1, \ldots, N-1 \}$ the limit
	\[
		\calR_i = \lim_{n \to \infty} R_{nN+i}(x)
	\]	
	exists where 
	\[
		R_n = a_{n+N-1} \big( X_n - \sigma \Id \big).
	\]
	Let $K \subset \Lambda$ be a compact interval with non-empty interior, where $\Lambda$ is defined 
	in \eqref{eq:32}. If 
	\[
		(X_n : n \in \NN), (R_n : n \in \NN) \in \calD_1^N \big( K, \GL(2, \RR) \big),
	\]	
	then
	\[
		\lim_{n \to \infty}
		\frac{1}{\rho_n} K_n\bigg(x + \frac{u}{\rho_n}, x + \frac{v}{\rho_n} \bigg)
		=
		\frac{\upsilon(x)}{\mu'(x)}
		\sinc \big( (u-v) \pi \upsilon(x) \big)
	\]
	locally uniformly with respect to $x \in \Lambda$ and $u, v \in \RR$, where 
	\[
		\rho_n = \sum_{j = 0}^n \frac{\alpha_j}{a_j},
	\]
	and $\upsilon$ is defined in \eqref{eq:33}.
\end{theorem}
\begin{proof}
	Let $K$ be a compact interval with non-empty interior contained in $\Lambda$ and let $L > 0$. We select a
	compact interval $\tilde{K} \subset \Lambda$ containing $K$ in its interior. There is $n_0 > 0$ such that
	for all $x \in K$, $n \geq n_0$, $i \in \{0, 1, \ldots, N-1\}$, and $u \in [-L, L]$,
	\[
		x + \frac{u}{\rho_{n N + i}}, x + \frac{u}{N \alpha_i \rho_{i; n}} \in \tilde{K}.
	\]
	Given $x \in K$ and $u, v \in [-L, L]$, we set
	\begin{align*}
		x_{i; n} &= x + \frac{u}{N \alpha_i \rho_{i; n}}, \qquad x_{nN+i} = x + \frac{u}{\rho_{nN+i}}, \\
		y_{i; n} &= x + \frac{v}{N \alpha_i \rho_{i; n}}, \qquad y_{nN+i} = x + \frac{v}{\rho_{nN+i}}.
	\end{align*}
	In view of Remark \ref{rem:1},
	\[
		(X_{jN+i} : j \in \NN), (R_{jN+i} : j \in \NN) \in \calD_1\big(K, \GL(2, \RR)\big).
	\]
	Hence, by Theorem \ref{thm:5}, there are $c > 0$ and $M \in \NN$ such that for all $x, y \in K$, and $k \geq M$,
	\[
		\begin{aligned}
		&a_{(k+1)N+i-1} p_{kN+i}(x) p_{kN+i}(y) \\
		&\qquad\qquad= \frac{2}{\pi}
		\sqrt{\frac{|[\calR_i(x)]_{2,1}]|}{\mu'(x) \sqrt{-\discr \calR_i(x)}}}
		\sqrt{\frac{|[\calR_i(y)]_{2,1}|}{\mu'(y) \sqrt{-\discr \calR_i(y)}}} \\
		&\qquad\qquad\phantom{=}\times
		\sin\Big(\sum_{j=M+1}^k \theta_{jN+i}(x) + \eta_i(x) \Big)
		\sin\Big(\sum_{j=M+1}^k \theta_{jN+i}(y) + \eta_i(y) \Big)
		+
		E_{kN+i}(x, y)
		\end{aligned}
	\]
	where
	\[
		\sup_{x, y \in K} |E_{kN+i}(x, y)|
		\leq
		c\sum_{j = k}^\infty \sup_K \|X_{(j+1)N+i} - X_{jN+i}\| + \sup_K \|R_{(j+1)N+i} - R_{jN+i}\|.
	\]
	Therefore, we obtain
	\[
		\begin{aligned}
		&
		\sum_{k = M}^n
		p_{kN+i}(x) p_{kN+i}(y)
		=
		\frac{2}{\pi} 
		\sqrt{\frac{|[\calR_i(x)]_{2,1}]|}{\mu'(x) \sqrt{-\discr \calR_i(x)}}}
		\sqrt{\frac{|[\calR_i(y)]_{2,1}|}{\mu'(y) \sqrt{-\discr \calR_i(y)}}} \\
		&\qquad\qquad\phantom{=}\times
		\sum_{k = M}^n
		\frac{1}{a_{(k+1)N+i-1}} 
		\sin\Big(\sum_{j=M+1}^k \theta_{jN+i}(x) + \eta_i(x) \Big)
		\sin\Big(\sum_{j=M+1}^k \theta_{jN+i}(y) + \eta_i(y) \Big) \\
		&\qquad\qquad\phantom{=}+
		\sum_{k = M}^n
		\frac{1}{a_{(k+1)N+i-1}} E_{kN+i}(x, y).
		\end{aligned}
	\]
	Observe that by the Stolz--Ces\'aro theorem,
	\[
		\lim_{n \to \infty} \frac{1}{\rho_{i-1; n}} \sum_{k = M+1}^n \frac{1}{a_{(k+1)N+i-1}} E_{kN+i}(x, y)
		=
		\lim_{n \to \infty} \frac{a_{nN+i-1}}{a_{(n+1)N+i-1}} E_{nN+i}(x, y) = 0.
	\]
	In view of Proposition \ref{prop:4}, we can apply Theorem \ref{thm:11} with
	\[
		\xi_j(x) = \theta_{jN+i}(x),
		\qquad
		\gamma_j = \frac{N \alpha_{i-1}}{a_{(j+1)N+i-1}},
		\qquad\text{and}\qquad
		|\psi(x)| = \pi \upsilon(x).
	\]
	Therefore, for any $i' \in \{0, 1, \ldots, N-1\}$, as $n$ tends to infinity
	\begin{align*}
		\frac{1}{N \alpha_{i-1} \rho_{i-1; n}}
		\sum_{k = M}^n
		\frac{N \alpha_{i-1}}{a_{(k+1)N+i-1}}
		&\sin\Big(\sum_{j=M+1}^k \theta_{jN+i}(x_{nN+i'}) + \eta_i(x_{nN+i'}) \Big) \\
		&\times
		\sin\Big(\sum_{j=M+1}^k \theta_{jN+i}(y_{nN+i'}) + \eta_i(y_{nN+i'}) \Big)
	\end{align*}
	approaches to
	\[
		\sinc\big((v-u) \pi \upsilon(x) \big)
	\]
	uniformly with respect to $x \in K$ and $u, v \in [-L, L]$. Moreover,
	\begin{align*}
		\lim_{n \to \infty} 
		\frac{|[\calR_i(x_{nN+i'})]_{2,1}]|}{\mu'(x_{nN+i'}) \sqrt{-\discr \calR_i(x_{nN+i'})}}
		&=
		\lim_{n \to \infty}
		\frac{|[\calR_i(y_{nN+i'})]_{2,1}]|}{\mu'(y_{nN+i'}) \sqrt{-\discr \calR_i(y_{nN+i'})}} \\
		&=
		\frac{|[\calR_i(x)]_{2,1}]|}{\mu'(x) \sqrt{-\discr \calR_i(x)}}.
	\end{align*}
	Hence,
	\begin{equation}
		\label{eq:90}
		\lim_{n \to \infty} 
		\frac{1}{\rho_{i-1; n}}
		K_{i; n}(x_{nN+i'}, y_{nN+i'}) 
		=
		\sinc\big((v-u)\pi \upsilon(x) \big)
		\cdot
		\frac{|[\calR_i(x)]_{2,1}]|}{\pi \mu'(x) \sqrt{-\discr \calR_i(x)}}.
	\end{equation}
	Finally, we write
	\[
		K_{nN+i'}(x, y) = \sum_{i = 0}^{N-1} K_{i; n}(x, y) 
		+ \sum_{i = i'+1}^{N-1} \big(K_{i; n-1}(x, y) - K_{i; n}(x, y)\big).
	\]
	Observe that
	\[
		\sup_{x, y \in K}{\big|K_{i; n-1}(x, y) - K_{i; n}(x, y)\big|} =
		\sup_{x, y \in K}{|p_{nN+i}(x) p_{nN+i}(y)|} \leq c,
	\]
	thus, by \eqref{eq:90} and \eqref{eq:78},
	\begin{align*}
		\lim_{n \to \infty} \frac{1}{\rho_{nN+i'}} K_{nN+i'}(x_{nN+i'}, y_{nN+i'})
		&=
		\lim_{n \to \infty} \sum_{i = 0}^{N-1} \frac{1}{\rho_{i-1; n}} K_{nN+i}(x_{nN+i'} , y_{nN+i'}) \cdot 
		\frac{\rho_{i-1; n}}{\rho_{nN+i'}} \\
		&=
		\frac{1}{\mu'(x)}
		\sinc\big((v-u)\pi \upsilon(x) \big)
		\frac{1}{N \pi}
		\sum_{i = 0}^{N-1}
		\frac{|[\calR_{i}(x)]_{2,1}]|}{\sqrt{-\discr \calR_{i}(x)}} \cdot \frac{1}{\alpha_{i-1}}.
	\end{align*}
	Hence, by \eqref{eq:33},
	\[
		\lim_{n \to \infty} \frac{1}{\rho_{nN+i'}} K_{nN+i'}(x_{nN+i'}, y_{nN+i'})
		=
		\frac{\upsilon(x)}{\mu'(x)} \sinc\big((v-u) \pi \upsilon(x) \big),
	\]
	and the theorem follows.
\end{proof}

\begin{bibliography}{jacobi}
	\bibliographystyle{amsplain}
\end{bibliography}

\end{document}